\documentclass[12pt,a4paper]{amsart}
\usepackage{amssymb,amsxtra,eucal}
\usepackage[all,cmtip]{xy}
\usepackage{tikz-cd}

\calclayout

\newcommand{\+}{\nobreakdash-}
\renewcommand{\:}{\colon}

\newcommand{\rarrow}{\longrightarrow}

\newcommand{\ot}{\otimes}

\DeclareMathOperator{\Spec}{Spec}
\DeclareMathOperator{\Hom}{Hom}
\DeclareMathOperator{\Ext}{Ext}
\DeclareMathOperator{\Tor}{Tor}

\newcommand{\modl}{{\operatorname{\mathsf{--mod}}}}
\newcommand{\contra}{{\operatorname{\mathsf{--contra}}}}
\newcommand{\qcoh}{{\operatorname{\mathsf{--qcoh}}}}
\newcommand{\ctrh}{{\operatorname{\mathsf{--ctrh}}}}

\newcommand{\tors}{{\operatorname{\mathsf{-tors}}}}
\newcommand{\ctra}{{\operatorname{\mathsf{-ctra}}}}
\newcommand{\com}{{\operatorname{\mathsf{-com}}}}

\newcommand{\inj}{{\mathsf{inj}}}
\newcommand{\proj}{{\mathsf{proj}}}

\newcommand{\lrarrow}{\mskip.5\thinmuskip\relbar\joinrel\relbar
   \joinrel\rightarrow\mskip.5\thinmuskip\relax}
\newcommand{\llarrow}{\mskip.5\thinmuskip\leftarrow\joinrel\relbar
   \joinrel\relbar\mskip.5\thinmuskip}
\newcommand{\bu}{{\text{\smaller\smaller$\scriptstyle\bullet$}}}

\newcommand{\qs}{{\mathsf{qs}}}
\newcommand{\sep}{{\mathsf{sep}}}
\newcommand{\bs}{{\mathsf{b}}}
\newcommand{\as}{{\mathsf{a}}}

\newcommand{\R}{\widehat R}
\newcommand{\I}{\widehat I}
\newcommand{\fP}{\mathfrak P}
\newcommand{\C}{\mathcal C}
\newcommand{\F}{\mathcal F}
\newcommand{\cO}{\mathcal O}

\newcommand{\sA}{\mathsf A}
\newcommand{\sB}{\mathsf B}
\newcommand{\sC}{\mathsf C}
\newcommand{\sD}{\mathsf D}
\newcommand{\sT}{\mathsf T}

\newcommand{\boZ}{\mathbb Z}
\newcommand{\boQ}{\mathbb Q}
\newcommand{\boL}{\mathbb L}
\newcommand{\boR}{\mathbb R}

\newcommand{\s}{\mathbf s}

\newcommand{\Section}[1]{\bigskip\section{#1}\medskip}
\theoremstyle{plain}
\newtheorem{thm}{Theorem}[section]
\newtheorem{prop}[thm]{Proposition}
\newtheorem{lem}[thm]{Lemma}

\newtheorem{cor}[thm]{Corollary}

\theoremstyle{definition}
\newtheorem{rem}[thm]{Remark}
\newtheorem{exs}[thm]{Examples}

\begin{document}

\title{Remarks on derived complete modules \\
and complexes}

\author{Leonid Positselski}

\address{Institute of Mathematics of the Czech Academy of Sciences \\
\v Zitn\'a~25, 115~67 Prague~1 \\ Czech Republic; and
\newline\indent Laboratory of Algebra and Number Theory \\
Institute for Information Transmission Problems \\
Moscow 127051 \\ Russia} 

\email{positselski@math.cas.cz}

\begin{abstract}
 Let $R$ be a commutative ring and $I\subset R$ a finitely generated
ideal.
 We discuss two definitions of derived $I$\+adically complete
(also derived $I$\+torsion) complexes of $R$\+modules which appear
in the literature: the idealistic and the sequential one.
 The two definitions are known to be equivalent for a weakly proregular
ideal~$I$; we show that they are different otherwise.
 We argue that the sequential approach works well, but the idealistic
one needs to be reinterpreted or properly understood.
 We also consider $I$\+adically flat $R$\+modules.
\end{abstract}

\maketitle

\tableofcontents

\section*{Introduction}
\medskip

\setcounter{subsection}{-1}
\subsection{{}} \label{introd-how-many-subsecn}
 Let $I$ be a finitely generated ideal in a commutative ring~$R$.
 How many abelian categories of $I$\+adically complete (in some sense)
$R$\+modules are there?
 An unsuspecting reader would probably guess that there are none.
 In fact, generally speaking, there are two such abelian categories, one
of them a full subcategory in the other.
 For a Noetherian ring $R$, the two categories coincide.

 Furthermore, how many triangulated categories of derived $I$\+adically
complete complexes are there?
 We argue that the correct answer is ``three'', two of which are
the important polar cases and the third one is kind of intermediate.
 There is also one abelian category of $I$\+torsion $R$\+modules and
two triangulated categories of derived $I$\+torsion complexes.
 These triangulated categories are connected by natural triangulated
functors.
 For a Noetherian ring $R$ (or more generally, for a weakly proregular
ideal~$I$), these functors are equivalences; so the answer to all
the ``how many'' questions reduces to ``one''.

 To be precise, the definitions of derived complete and torsion
complexes that can be found in the literature do not always agree with
our suggested definitions.
 We discuss both, and explain why we view some of our constructions of
the triangulated categories of derived complete and torsion complexes as
improvements upon the ones previously considered by other authors.

 This paper is inspired by Yekutieli's paper~\cite{Yek3} (as well as his
earlier paper~\cite{Yek2}), where some of the results of the present
author have been mentioned.
 The credit is due to Yekutieli for posing several questions to which
the present paper provides the answers.

\subsection{{}} \label{introd-idealistic-completion-subsecn}
 Let us briefly discuss the substantial issue involved, starting
for simplicity with the particular case of a principal ideal $I=(s)
\subset R$ generated by an element $s\in R$.
 Let $M$ be an $R$\+module.
 What is the derived $s$\+adic completion of~$M$?

 A more na\"\i ve approach is to start with the underived $I$\+adic
completion,
$$
 \Lambda_s(M)=\varprojlim\nolimits_{n\ge1}M/s^nM.
$$
 The problem with the functor $\Lambda_s$ is that it is the composition
of a right exact functor assigning to $M$ the system of its quotient
modules $\dotsb\rarrow M/s^3M\rarrow M/s^2M\rarrow M/sM$ with the left
exact functor of projective limit.
 As such, the functor $\Lambda_s$ is neither left nor right exact, and
in fact not even exact in the middle~\cite{Yek0}.

 Nevertheless, one can consider the left derived functor
$\boL_*\Lambda_s$ of $\Lambda_s$, computable by applying $\Lambda_s$
to a projective resolution of an $R$\+module~$M$.
 This is equivalent to replacing $\Lambda_s$ with its better behaved
(right exact) $0$\+th left derived functor $\boL_0\Lambda_s$ and
computing the left derived functor of $\boL_0\Lambda_s$.
 Then one can say that the derived $I$\+adic completion of $M$ is
the complex
$$
 \boL\Lambda_s(M)=\boL(\boL_0\Lambda_s)(M)=\Lambda_s(P_\bu),
$$
where $P_\bu$ is a projective resolution of~$M$.
 In fact, it suffices to require $P_\bu$ to be a flat resolution.
 This approach is taken, in the generality of finitely generated
ideals $I\subset R$, by Porta--Shaul--Yekutieli in the paper~\cite{PSY}
(see also the much earlier~\cite[Section~1]{GM}).
 It is called the ``idealistic derived adic completion'' in~\cite{Yek3}.

\subsection{{}} \label{introd-sequential-completion-subsecn}
 A more sophisticated approach is to construct the derived $I$\+adic
completion as the derived functor of projective limit of the derived
functors of the passage to the quotient module $M\longmapsto M/s^nM$.
 Furthermore, the latter derived functors are interpreted as taking
an $R$\+module $M$ to the two-term complexes $M\overset{s^n}\rarrow M$,
concentrated in the cohomological degrees~$-1$ and~$0$.
 Then the complex $\boR\varprojlim_{n\ge1}(M\overset{s^n}\rarrow M)$
can be computed as
$$
 \boR\varprojlim\nolimits_{n\ge1}(M\overset{s^n}\rarrow M) =
 \boR\Hom_R(R\to R[s^{-1}],\>M).
$$
 This construction can be found in the paper~\cite[Section~1.1]{Beil};
see also~\cite[Section~1.5]{KS} and~\cite[Section~3.4]{BS}
(once again, the much earlier paper~\cite[Section~2]{GM} is relevant).
 This approach also extends naturally to finitely generated ideals
$I\subset R$ \,\cite{GM,BS}, for which it is helpful to choose a finite
sequence of generators $s_1$,~\dots, $s_m\in I$, and then show
that the construction does not depend on a chosen set of generators.
 For this reason, it is called the ``sequential derived completion''
in~\cite{Yek3}.

\subsection{{}}
 One important difference between the two approaches is that
the idealistic derived completion is sensitive to the choice of
a base ring~$R$.
 Given a subring $R'\subset R$ such that $s\in R'$, the derived functors
$\boL_R\Lambda_s$ and $\boL_{R'}\Lambda_s$ computed in the categories
of $R$\+modules and $R'$\+modules do \emph{not} agree with each other
in any meaningful way, generally speaking.
 The sequential derived completion construction, on the other hand,
essentially happens over the ring $\boZ[s]$.
 In particular, for any commutative ring homomorphism $R'\rarrow R$
and an element $s'\in R'$, denoting by~$s$ the image of $s'$ in $R$,
for any $R$\+module $M$ one has a natural isomorphism
$$
 \boR\Hom_R(R\to R[s^{-1}],\>M)\,\simeq\,
 \boR\Hom_{R'}(R'\to R'[s'{}^{-1}],\>M)
$$
in the derived category of $R'$\+modules.

\subsection{{}} \label{introd-derived-torsion-subsecn}
 Torsion is easier to think of than completion; but the same
difference between two approaches arises.
 Denote by $\Gamma_s(M)\subset M$ the submodule consisting of all
the \emph{$s$\+torsion} elements, i.~e., elements $m\in M$ for which
there exists an integer $n\ge1$ such that $s^nm=0$.
 The functor $\Gamma_s(M)$ is left exact, so there is no problem
involved in considering its right derived functor $\boR\Gamma_s$,
computable as
$$
 \boR\Gamma_s(M)=\Gamma_s(J^\bu),
$$
where $J^\bu$ is an injective resolution of~$M$.
 Following the terminology of~\cite{Yek3}, this is the ``idealistic
derived torsion functor''.

 Alternatively, one can consider the functors $M\longmapsto{}_{s^n}M$,
assigning to an $R$\+module $M$ its submodule of all elements
annihilated by~$s^n$.
 Then one can interpret the right derived functor of $M\longmapsto
{}_{s^n}M$ as taking an $R$\+module $M$ to the two-term complex
$M\overset{s^n}\rarrow M$, concentrated in the cohomological
degrees~$0$ and~$1$.
 Finally, the derived $I$\+torsion of $M$ is defined as the inductive
limit
$$
 \varinjlim\nolimits_{n\ge1}(M\overset{s^n}\rarrow M) =
 (R\to R[s^{-1}])\ot_R M.
$$
 This is called the ``sequential derived torsion functor''
in~\cite{Yek3}.

 Similarly to the two derived completions, the two derived torsions
differ in how they behave with respect to the ring changes.
 The idealistic derived torsion is sensitive to the choice of
a base ring $R$, while the sequential derived torsion is
indifferent to it.

\subsection{{}}
 Let us say a few words about the \emph{weak proregularity} condition,
which plays a central role in our discussion.
 A principal ideal $I=(s)\subset R$ is weakly proregular if and only if
the \emph{$s$\+torsion} in $R$ is \emph{bounded}, i.~e., there exists
an integer $n_0\ge1$ such that $s^nr=0$ for some $n\ge1$ and $r\in R$
implies $s^{n_0}r=0$.

 For any $R$\+module $M$ there is a natural morphism
$$
 \boR\Gamma_s(M)\lrarrow(R\to R[s^{-1}])\ot_R M
$$
in the derived category of $R$\+modules.
 It turns out that this morphism is an isomorphism for all $R$\+modules
$M$ if and only if the $s$\+torsion in $R$ is bounded.
 It suffices to check this condition for injective $R$\+modules $M=J$.
 A version of this result mentioning the weak proregularity in place of
the bounded torsion holds for any finitely generated ideal $I\subset R$
\,\cite[Lemma~2.4]{LC}, \cite[Theorem~3.2]{Sch}, \cite[Theorem~4.24]{PSY}.

 Similarly, for any $R$\+module $M$ there is a natural morphism
$$
 \boR\Hom_R(R\to R[s^{-1}],\>M)\lrarrow\boL\Lambda_s(M)
$$
in the derived category of $R$\+modules.
 Once again, this morphism is an isomorphism for all $R$\+modules $M$ if
and only if the $s$\+torsion in $R$ is bounded.
 It suffices to check this condition for the free $R$\+module with
a countable set of generators $M=\bigoplus_{n=1}^\infty R$.
 A version of this result with the bounded torsion replaced by
 weak proregularity holds for any finitely generated ideal $I\subset R$
\cite[Remark~7.8]{HS}.

\subsection{{}}
 Various triangulated categories of derived complete and torsion
modules discussed in this paper are related, in one way or another,
to the two derived completion and two derived torsion functors
mentioned in Sections~\ref{introd-idealistic-completion-subsecn},
\ref{introd-sequential-completion-subsecn}
and~\ref{introd-derived-torsion-subsecn}.

\subsection{{}}
 Before we finish this introduction, let us say a few words about our
motivation.
 The objects of the two abelian categories of derived $I$\+adically
complete $R$\+modules mentioned in Section~\ref{introd-how-many-subsecn}
are called \emph{$I$\+contramodule $R$\+modules} and \emph{quotseparated
$I$\+contramodule $R$\+modules} in this paper.
 The following example is illuminating.

 Let $R=k[x_1,\dotsc,x_m]$ be the ring of polynomials in a finite
number of variables over a field~$k$ and $I=(x_1,\dotsc,x_m)\subset R$
be the ideal generated by the elements~$x_j$.
 Let $\C$ be the coassociative, cocommutative, counital coalgebra
over~$k$ such that the dual topological algebra
$\C^*=k[[x_1,\dotsc,x_m]]$ is the algebra of formal Taylor power series
in the variables~$x_j$.
 Then the abelian category of $I$\+torsion $R$\+modules is equivalent
to the category of $\C$\+comodules.
 Furthermore, the abelian category of $I$\+contramodule $R$\+modules
(which coincides with the abelian category of quotseparated
$I$\+contra\-module $R$\+modules in this case, as the ring $R$ is
Noetherian) is equivalent to the abelian category of $\C$\+contramodules.
 The latter category was defined by Eilenberg and Moore
in~\cite[Section~III.5]{EM}; we refer to our overview~\cite{Prev}
for a discussion.

 For any finitely generated ideal $I$ in a commutative ring $R$,
the category of $I$\+torsion $R$\+modules is a Grothendieck abelian
category.
 On the other hand, both the abelian categories of $I$\+contramodule
$R$\+modules and quotseparated $I$\+contramodule $R$\+modules are
locally presentable abelian categories with enough projective objects.
 We believe that the latter class of abelian categories, which are
dual-analogous or ``covariantly dual'' to Grothendieck abelian
categories~\cite{PR,Pper}, is not receiving the attention that
it deserves.
 Thus we use this opportunity to point out and discuss two classes
of examples of locally presentable abelian categories with enough
projectives appearing naturally in commutative algebra.

\subsection{{}} \label{introd-notation}
 Let us make some general notational conventions.
 Throughout this paper, $R$ is a commutative ring and $I\subset R$
is a finitely generated ideal.
 When we need to choose a finite set of generators of
the ideal $I$, we denote them by $s_1$,~\dots, $s_m\in I$.
 The sequence of elements $s_1$,~\dots, $s_m$ is denoted by~$\s$
for brevity.

 Given an abelian (or Quillen exact) category $\sA$, we denote its
bounded and unbounded derived categories by $\sD^\bs(\sA)$,
$\sD^+(\sA)$, $\sD^-(\sA)$, and $\sD(\sA)$, as usual.
 The abelian category of (arbitrarily large)
$R$\+modules is denoted by $R\modl$.

\subsection*{Acknowledgement}
 I~am grateful to Amnon Yekutieli for sharing an early version of his
paper~\cite{Yek3}, as well as for kindly including an advertisement
of a result of mine in~\cite[Remark~4.12]{Yek2}.
 I~also wish to thank two anonymous referees for many helpful suggestions. 
 The author is supported by the GA\v CR project 20-13778S and
research plan RVO:~67985840.

\Section{Derived Complete Modules} \label{derived-complete-modules-secn}

 The definition of an \emph{Ext-$p$-complete} (abelian, or more
generally, nilpotent) \emph{group} for a prime number~$p$ goes back to
the book of Bousfield and Kan~\cite[Sections~VI.2\+-4]{BK}.
 Under the name of \emph{weakly $l$\+complete} abelian groups (where
$l$~is still a prime number), they were discussed by Jannsen
in~\cite[Section~4]{Jan}.
 Even earlier, the abelian groups decomposable as products of
Ext-$p$-complete abelian groups over the prime numbers~$p$ were studied
by Harrison in~\cite[Section~2]{Harr} under the name of ``co-torsion
abelian groups''; this approach was generalized by Matlis in~\cite{Mat}.
 In our terminology, Ext-$p$-complete abelian groups are called
\emph{$p$\+contramodule $\boZ$\+modules} (we refer to the introductions
to the papers~\cite{Pcta,PMat} for a discussion).

 Given an element $s\in R$, consider the ring $R[s^{-1}]$ obtained
by adjoining to $R$ an element inverse to~$s$.
 Denoting by $S\subset R$ the multiplicative subset
$S=\{1,s,s^2,\dotsc\}$, one has $R[s^{-1}]=S^{-1}R$.
 One can easily see that the projective dimension of the $R$\+module
$R[s^{-1}]$ does not exceed~$1$.

 An $R$\+module $C$ is said to be
an \emph{$s$\+contramodule} if
$$
 \Hom_R(R[s^{-1}],C)=0=
 \Ext^1_R(R[s^{-1}],C).
$$
 A weaker condition is sometimes useful: an $R$\+module $C$ is said to
be \emph{$s$\+contra\-ad\-justed} if $\Ext^1_R(R[s^{-1}],C)=0$.

 An $R$\+module $C$ is said to be an \emph{$I$\+contramodule} (or
an \emph{$I$\+contramodule $R$\+module}) if $C$ is
an $s$\+contramodule for every element $s\in I$.
 It suffices to check this condition for any chosen set of generators
$s=s_1$, $s_2$,~\dots,~$s_m$ of the ideal $I\subset R$
\cite[Theorem~5.1]{Pcta}.

\begin{lem} \label{I-contramodule-R-modules-lemma}
 The full subcategory of $I$\+contramodule $R$\+modules
$R\modl_{I\ctra}$ is closed under the kernels, cokernels, extensions,
and infinite products (hence also under all limits) in
the category of $R$\+modules $R\modl$.
 Consequently, the category $R\modl_{I\ctra}$ is abelian and
the inclusion functor $R\modl_{I\ctra}\rarrow R\modl$ is exact.
\end{lem}

\begin{proof}
 This is~\cite[Proposition~1.1]{GL} or~\cite[Theorem~1.2(a)]{Pcta}.
\end{proof}

 Denoting by $\overline R\subset R$ the subring in $R$ generated by
the elements $s_1$,~\dots, $s_m$ over $\boZ$ and by
$\overline I\subset \overline R$ the ideal in $\overline R$ generated
by the same elements, one observes that an $R$\+module $C$ is
an $I$\+contramodule if and only if its underlying
$\overline R$\+module $C$ is an $\overline I$\+contramodule.
 In this specific sense, the definition of $I$\+contramodule
$R$\+modules reduces to the case of Noetherian rings.
 Moreover, one can replace $\overline R$ by the polynomial ring
$\widetilde R=\boZ[s_1,\dotsc,s_m]$ with the ideal $\widetilde I
=(s_1,\dotsc,s_m)$; an $R$\+module is an $I$\+contramodule if and
only if its underlying $\widetilde R$\+module is
an $\widetilde I$\+contramodule.

 The interpretation of the $s$\+contramodule $R$\+modules as
the \emph{$R$\+modules with $s$\+power infinite summation operation}
and the $I$\+contramodule $R$\+modules as the \emph{$R$\+modules
with $[s_1,\dotsc,s_m]$\+power infinite summation operation} is
discussed in~\cite[Sections~3\+-4]{Pcta} (see
also~\cite[Appendix~B]{Pweak} for the Noetherian case).

 The $I$\+adic completion of an $R$\+module $C$ is defined as
$$
 \Lambda_I(C)=\varprojlim\nolimits_{n\ge1}C/I^nC.
$$
 An $R$\+module $C$ is said to be \emph{$I$\+adically separated}
if the canonical morphism $C\rarrow\Lambda_I(C)$ is injective,
and we say that $C$ is \emph{$I$\+adically complete} if the morphism
$C\rarrow\Lambda_I(C)$ is surjective.

 Any $I$\+adically complete and separated $R$\+module is
an $I$\+contramodule.
 Any $I$\+contramodule $R$\+module is $I$\+adically
complete~\cite[Theorem~5.6]{Pcta}, but it need not be
$I$\+adically separated.

 An $I$\+contramodule $R$\+module is said to be \emph{quotseparated}
if it is a quotient $R$\+module of an $I$\+adically separated and
complete $R$\+module.
 We denote the full subcategory of $I$\+adically separated and
complete $R$\+modules by $R\modl_{I\ctra}^\sep\subset R\modl$
and the full subcategory of quotseparated $I$\+contramodule
$R$\+modules by $R\modl_{I\ctra}^\qs\subset R\modl$.
 So there are inclusions of full subcategories {\hbadness=1400
$$
 R\modl_{I\ctra}^\sep\,\subset\, R\modl_{I\ctra}^\qs\,\subset\,
 R\modl_{I\ctra}\,\subset\, R\modl.
$$
 An example of a quotseparated, but not separated $I$\+contramodule
$R$\+module can be found already in~\cite[Example~3.20]{Yek0}
or~\cite[Example~2.5]{Sim} (see also~\cite[Example~4.33]{PSY}).
 The additive category $R\modl_{I\ctra}^\sep$ is
not abelian~\cite[Example~2.7\,(1)]{Pcta}; but
the category $R\modl_{I\ctra}^\qs$ is, as the next lemma tells.

\begin{lem}
 The full subcategory of quotseparated $I$\+contramodule $R$\+modules
$R\modl_{I\ctra}^\qs$ is closed under subobjects, quotient objects,
and infinite products in $R\modl_{I\ctra}$, and closed under
the kernels, cokernels, and infinite products (hence also under all
limits) in $R\modl$.
 Consequently, the category $R\modl_{I\ctra}^\qs$ is abelian and
the inclusion functors $R\modl_{I\ctra}^\qs\rarrow R\modl_{I\ctra}
\rarrow R\modl$ are exact.
\end{lem}

\begin{proof}
 The basic observation is that any submodule of an $I$\+adically
separated $R$\+module is $I$\+adically separated.
 Hence the class $\sC$ of all $I$\+adically separated $I$\+contramodules
(\,$=$~$I$\+adically separated and complete $R$\+modules) is
closed under subobjects in $R\modl_{I\ctra}$.
 Moreover, the class of all $I$\+adically separated $R$\+modules is
closed under products in $R\modl$; hence the class $\sC$ is closed
under products in the abelian category $\sA=R\modl_{I\ctra}$.

 Now for any abelian category $\sA$ and a class of objects $\sC
\subset\sA$ such that $\sC$ is closed under subobjects in $\sA$,
the class $\sB$ of all quotient objects of objects from $\sC$ is closed
under subobjects \emph{and} quotient objects in~$\sA$.
 Hence $\sB$ is an abelian category with an exact inclusion functor
$\sB\rarrow\sA$.
 If, moreover, the class $\sC$ is closed under products in $\sA$ and
the product functors in $\sA$ are exact, then the full subcategory
$\sB$ is closed under products in~$\sA$.

 In the situation at hand, these observations are applicable to
the abelian category $\sA=R\modl_{I\ctra}$, producing the abelian
category $\sB=R\modl_{I\ctra}^\qs$.
 Then the remaining assertions of the lemma follow from
Lemma~\ref{I-contramodule-R-modules-lemma}.
\end{proof}

 Both} the full subcategories $R\modl_{I\ctra}$ and
$R\modl_{I\ctra}^\qs$ are reflective in $R\modl$ (i.~e., their inclusion
functors have left adjoints).
 The reflector $\Delta_I\:R\modl\rarrow R\modl_{I\ctra}$ is constructed
and discussed at length in~\cite[Sections~6\+-7]{Pcta}
(see formula~\eqref{delta-computed} in Section~\ref{wpr-secn} below).
 The functor $\Delta_I\:R\modl\rarrow R\modl_{I\ctra}$ is right exact,
because it is a left adjoint; since the inclusion $R\modl_{I\ctra}
\rarrow R\modl$ is an exact functor, the composition $R\modl\rarrow
R\modl_{I\ctra}\rarrow R\modl$ is also right exact.
 As any reflector onto a full subcategory, the functor $\Delta_I\:
R\modl\rarrow R\modl$ is idempotent, and its essential image is
the whole full subcategory $R\modl_{I\ctra}\subset R\modl$.

 As pointed out in~\cite[Section~1]{Yek0} and~\cite[Section~1]{Yek3},
the $I$\+adic completion functor $\Lambda_I\:R\modl\rarrow R\modl$ is
neither left, nor right exact (even though, by~\cite[Theorem~5.8]{Pcta},
\,$\Lambda_I\:R\modl\rarrow R\modl_{I\ctra}^\sep$ is the reflector onto
the full subcategory of $I$\+adically separated and complete modules
in $R\modl$).
 The next proposition describes the reflector
$R\modl\rarrow R\modl_{I\ctra}^\qs$ as the $0$\+th left derived functor
of the functor~$\Lambda_I$, that is,
$\boL_0\Lambda_I\:R\modl\rarrow R\modl_{I\ctra}^\qs$.

\begin{prop} \label{quotseparated-reflector}
\textup{(a)} For any $R$\+module $C$, there are two natural surjective
$R$\+module morphisms\/ $\Delta_I(C)\rarrow\boL_0\Lambda_I(C)\rarrow
\Lambda_I(C)$. \par
\textup{(b)} For any $R$\+module $C$, the $R$\+module\/
$\boL_0\Lambda_I(C)$ is a quotseparated $I$\+contramodule. \par
\textup{(c)} The functor\/ $\boL_0\Lambda_I\:R\modl\rarrow
R\modl_{I\ctra}^\qs$ is left adjoint to the inclusion
$R\modl_{I\ctra}^\qs\rarrow R\modl$. \par
\textup{(d)} Consequently, the functor $\boL_0\Lambda_I\:R\modl\rarrow
R\modl$ is idempotent, and its essential image is the full
subcategory $R\modl_{I\ctra}^\qs\subset R\modl$.
\end{prop}

\begin{proof}
 Part~(a): a natural surjective $R$\+module morphism $b_{I,C}\:
\Delta_I(C)\rarrow\Lambda_I(C)$ is constructed
in~\cite[Lemma~7.5]{Pcta}, and its kernel is also computed
in~\cite[Lemma~7.5]{Pcta}.
 Let us show that the map~$b_{I,C}$ factorizes naturally as
the composition of two surjective morphisms $\Delta_I(C)\rarrow
\boL_0\Lambda_I(C)\rarrow\Lambda_I(C)$.

 Let $P_1\rarrow P_0\rarrow C\rarrow0$ be a right exact sequence of
$R$\+modules with projective $R$\+modules $P_1$ and~$P_0$.
 By the definition, $\boL_0\Lambda_I(C)$ is the cokernel of
the induced morphism $\Lambda_I(P_1)\rarrow\Lambda_I(P_0)$.
 We will construct a commutative diagram of $R$\+module morphisms
of the following form:
\begin{equation} \label{delta-lambda-diagram}
\begin{gathered}
 \xymatrix{
  \Delta_I(P_1)\ar@{->>}[d]^{b_{I,P_1}} \ar[r] & \Delta_I(P_0)
  \ar@{->>}[d]^{b_{I,P_0}} \ar[r] & \Delta_I(C) \ar[r]\ar@{.>>}[d]
  & 0 \\
  \Lambda_I(P_1)\ar@{=}[d]\ar[r] & \Lambda_I(P_0) \ar@{=}[d]
  \ar[r] & \boL_0\Lambda_I(C) \ar[r]\ar@{.>>}[d] & 0 \\
  \Lambda_I(P_1)\ar[r] & \Lambda_I(P_0) \ar@{->>}[r] & 
  \Lambda_I(C)
 }
\end{gathered}
\end{equation}

 The uppper and lower rows are obtained by applying the functors
$\Delta_I$ and $\Lambda_I$, respectively, to the right exact
sequence $P_0\rarrow P_1\rarrow C\rarrow0$.
 The upper leftmost square is obtained by applying the natural
transformation~$b_I$ to the $R$\+module morphism $P_1\rarrow P_0$.
 The upper and middle rows are right exact sequences (as
the functor~$\Delta_I$ is right exact by the above discussion).
 The upper dotted vertical arrow is uniquely defined by
the condition of commutativity of the upper rightmost square.
 The morphism $\Delta_I(C)\rarrow\boL_0\Lambda_I(C)$ is surjective
because the morphism~$b_{I,P_0}$ is.

 The lower row is not exact in the middle, generally speaking;
but the morphism $\Lambda_I(P_0)\rarrow\Lambda_I(C)$ is still
an epimorphism, because the functor $\Lambda_I$ takes epimorphisms
to epimorphisms~\cite[Proposition~1.2]{Yek0}.
 The lower dotted vertical arrow is uniquely defined by
the condition of commutativity of the lower rightmost square.
 It follows that the morphism $\boL_0\Lambda_I(C)\rarrow
\Lambda_I(C)$ is surjective.

 Finally, the composition $\Delta_I(C)\rarrow\boL_0\Lambda_I(C)
\rarrow\Lambda_I(C)$ is equal to the morphism~$b_{I,C}$, since
both of them make the square diagram $\Delta_I(P_0)\rarrow
\Delta_I(C)\rarrow\Lambda_I(C)$, \ $\Delta_I(P_0)\rarrow\Lambda_I(P_0)
\rarrow\Lambda_I(C)$ commutative.

 Part~(b): $\boL_0\Lambda_I(C)$ is the cokernel of a morphism of
$I$\+adically separated and complete $R$\+modules $\Lambda_I(P_1)
\rarrow\Lambda_I(P_0)$.
 Any such a cokernel is a quotseparated $I$\+contramodule $R$\+module.
 
 Part~(c): in order to establish the adjunction, it remains to show
that, for any quotseparated $I$\+contramodule $R$\+module $K$ and
any $R$\+module morphism $C\rarrow K$, the induced morphism of
$I$\+contramodule $R$\+modules $\Delta_I(C)\rarrow K$ factorizes through
the epimorphism $\Delta_I(C)\rarrow\boL_0\Lambda_I(C)$.
 The upper rightmost square in the diagram~\eqref{delta-lambda-diagram}
is cocartesian, because the map~$b_{I,P_1}$ is surjective.
 Therefore, it suffices to check that the composition $\Delta_I(P_0)
\rarrow\Delta_I(C)\rarrow K$ factorizes through the epimorphism
$\Delta_I(P_0)\rarrow\Lambda_I(P_0)$.

 Let $L\rarrow K$ be a surjective $R$\+module morphism onto $K$ from
an $I$\+adically separated and complete $R$\+module~$L$.
 The object $\Delta_I(P_0)$ is projective in the abelian category
$R\modl_{I\ctra}$, because the functor $\Delta$, being left adjoint to
the exact inclusion functor, takes projectives to projectives.
 Both the $R$\+modules $L$ and $K$ are $I$\+contramodules; consequently,
the $R$\+module morphism $\Delta_I(P_0)\rarrow K$ factorizes through
the epimorphism $L\rarrow K$:
$$
 \xymatrix{
  \Delta_I(P_0) \ar@{..>}[d]\ar@{->>}[r] & \Delta_I(C) \ar[d] \\
  L \ar@{->>}[r] & K
 }
$$
 Finally, since $L\in R\modl_{I\ctra}^\sep$ and $\Lambda_I$ is
the reflector onto $R\modl_{I\ctra}^\sep$, any $R$\+module morphism
$\Delta_I(P)\rarrow L$ factorizes through the epimorphism
$b_{I,P}\:\Delta_I(P)\rarrow\Lambda_I(\Delta_I(P))=\Lambda_I(P)$
(for any $R$\+module~$P$).
 Simply put, $L$ is $I$\+adically separated and $\Lambda_I(P)$ is
the maximal $I$\+adically separated quotient $R$\+module of
$\Delta_I(P)$.

 Part~(d) follows from part~(c).
\end{proof}

 For any set $X$, let us denote by $R[X]=R^{(X)}$ the free $R$\+module
with generators indexed by~$X$.
 Then the $R$\+module $\Delta_I(R[X])$ is called the \emph{free
$I$\+contramodule $R$\+module with $X$ generators}, while
$\boL_0\Lambda_I(R[X])=\Lambda_I(R[X])$ is the \emph{free quotseparated
$I$\+contramodule $R$\+module with $X$ generators}.

 The $R$\+module $\Lambda_I(R[X])$ is what Yekutieli calls
the $R$\+module of \emph{decaying functions} $X\rarrow\Lambda_I(R)$
\cite[Section~2]{Yek0} (see also the much
earlier~\cite[Section~II.2.4.2]{RG}).
 The notation in~\cite{Yek0} is $\mathrm{F}_{\mathrm{fin}}(X,R)=R[X]$
(for the finitely supported functions $X\rarrow R$) and
$\mathrm{F}_{\mathrm{dec}}(X,\Lambda_I(R))=\Lambda_I(R[X])$
(for the decaying functions $X\rarrow\Lambda_I(R)$).

 There are enough projective objects in both the abelian categories
$R\modl_{I\ctra}$ and $R\modl_{I\ctra}^\qs$.
 The projective objects of the category $R\modl_{I\ctra}$ are precisely
the direct summands of the $R$\+modules $\Delta_I(R[X])$.
 The projective objects of the category $R\modl_{I\ctra}^\qs$ are
precisely the direct summands of the $R$\+modules $\Lambda_I(R[X])$.

\begin{lem} \label{when-projective-quotseparated}
\textup{(a)} Let $Q$ be a projective object of the abelian category
$R\modl_{I\ctra}$ (i.~e., the $R$\+module $Q$ is a direct summand of
the $R$\+module $\Delta_I(R[X])$ for some set~$X$).
 Then $Q$ is a quotseparated $I$\+contramodule $R$\+module if and only
if it is an $I$\+adically separated (and complete) $R$\+module. \par
\textup{(b)} For any fixed set $X$, the $R$\+module $\Delta_I(R[X])$
is quotseparated if and only if the natural map $b_{I,R[X]}\:
\Delta_I(R[X])\rarrow\Lambda_I(R[X])$ is an isomorphism.
\end{lem}

\begin{proof}
 Part~(a): suppose $Q$ is quotseparated; then it is a quotient module
of some $I$\+adically separated and complete $R$\+module~$C$.
 Since $C\rarrow Q$ is an epimorphism in the abelian category
$R\modl_{I\ctra}$ and $Q$ is a projective object in $R\modl_{I\ctra}$,
it follows that $Q$ is a direct summand of~$C$.
 Hence $Q$ is $I$\+adically separated.
 Part~(b): since $\Lambda_I(\Delta_I(P))=\Lambda_I(P)$ for any
$R$\+module $P$, the map $b_{I,P}\:\Delta_I(P)\rarrow\Lambda_I(P)$
is an isomorphism if and only if the $R$\+module $\Delta_I(P)$ is
$I$\+adically separated.
\end{proof}

 The $I$\+adic completion $\R=\Lambda_I(R)=\varprojlim_{n\ge1}R/I^n$
of the ring $R$ is a topological ring in the topology of projective
limit of discrete rings $R/I^n$ (which coincides with the $I$\+adic
topology of the $R$\+module~$\R$).
 Following the memoir~\cite[Section~1.2]{Pweak},
the paper~\cite[Sections~1.1\+-1.2 and~5]{PR}, or
the paper~\cite[Section~6]{PS1}, one can assign to a topological ring
$\R$ the abelian category of \emph{$\R$\+contramodules} $\R\contra$.
 These are modules with infinite summation operations with families
of coefficients converging to zero in~$\R$.

\begin{prop} \label{top-ring-contramodules}
 The forgetful functor\/ $\R\contra\rarrow R\modl$ (induced by
the canonical ring homomorphism $R\rarrow\R$) is fully faithful,
and its essential image is the full subcategory of quotseparated
$I$\+contramodule $R$\+modules $R\modl_{I\ctra}^\qs$.
 So we have an equivalence of abelian categories\/ $\R\contra\simeq
R\modl_{I\ctra}^\qs$.
\end{prop}

\begin{proof}
 This is a straightforward generalization of~\cite[Theorem~5.20]{PSl},
based on the discussion in~\cite[Example~3.6\,(3)]{Pper}.
 One can also observe that the forgetful functor $\R\contra\rarrow
R\modl$ takes the free $\R$\+contramodules $\R[[X]]$ to the free
quotseparated $I$\+contramodule $R$\+modules $\Lambda_I(R[X])$.
 Simply put, by the definition of an $\R$\+contramodule and in view of
Proposition~\ref{quotseparated-reflector}, both the abelian categories
$\R\contra$ and $R\modl_{I\ctra}^\qs$ are equivalent to the category of
modules over the same monad $X\longmapsto\R[[X]]=\Lambda_I(R[X])$ on
the category of sets.
\end{proof}

 In particular, it follows from Proposition~\ref{top-ring-contramodules}
that every quotseparated $I$\+contra\-module $R$\+module has a natural,
functorially defined $\R$\+module structure (extending the $R$\+module
structure).
 Analogously, one can define a commutative $R$\+algebra structure on
the $R$\+module $\Delta_I(R)$ and show that every $I$\+contramodule
$R$\+module has a functorially defined $\Delta_I(R)$\+module structure
(extending its $R$\+module structure) \cite[Example~5.2\,(3)]{Pper}.

 Notice that the full subcategory of quotseparated $I$\+contramodule
$R$\+modules does \emph{not} need to be closed under extensions in
$R\modl_{I\ctra}$ or in $R\modl$.
 In fact, the opposite is true.
 
\begin{prop} \label{extension-of-two}
 Every $I$\+contramodule $R$\+module is an extension of two
quotseparated $I$\+contramodule $R$\+modules.
\end{prop}

\begin{proof}
 This is explained in~\cite[Example~5.2\,(6)]{Pper}.
 The argument is based on the computation of the kernel of
the natural surjective morphism $b_{I,C}\:\Delta_I(C)\rarrow
\Lambda_I(C)$ in~\cite[Lemma~7.5]{Pcta}.
 Basically, any $I$\+contramodule $R$\+module $C$ is naturally
isomorphic to the $R$\+module $\Delta_I(C)$; and for any $R$\+module
$C$, the $R$\+module $\Lambda_I(C)$ is $I$\+adically separated and
complete, while the kernel of the map~$b_{I,C}$ can be described as
the cokernel of a natural map of $I$\+adically separated and complete
$R$\+modules.
\end{proof}

\begin{lem} \label{when-all-quotseparated}
 The full subcategories $R\modl_{I\ctra}^\qs$ and $R\modl_{I\ctra}$
coincide in $R\modl$ (i.~e., in other words, every $I$\+contramodule
$R$\+module is quotseparated) if and only if the map
$b_{I,R[X]}\:\Delta_I(R[X])\rarrow\Lambda_I(R[X])$ is
an isomorphism for every set~$X$.
 It suffices to check the latter condition for the countable set
$X=\omega$.
\end{lem}

\begin{proof}
 The first assertion is essentially a tautology;
see~\cite[Proposition~2.1]{Pper} for the details.
 The second assertion holds because both the functors $\Delta_I$
and $\Lambda_I$ preserve countably-filtered direct limits;
so, for any infinite set $X$, one has
$\Delta_I(R[X])=\varinjlim_{Z\subset X}\Delta_I(R[Z])$ and
$\Lambda_I(R[X])=\varinjlim_{Z\subset X}\Lambda_I(R[Z])$,
where $Z$ ranges over all the countably infinite subsets of~$X$.
\end{proof}

\begin{exs}
 Let us mention an explicit example of an $I$\+contramodule $R$\+module
which is not quotseparated.
 According to Lemma~\ref{when-projective-quotseparated}, for this
purpose it suffices to come up with an example of a commutative ring
$R$ with an ideal $I$ and a set $X$ such that the $R$\+module
$\Delta_I(R[X])$ is not $I$\+adically separated, or equivalently,
the natural morphism $b_{I,R[X]}\:\Delta_I(R[X])\rarrow\Lambda_I(R[X])$
is not an isomorphism.
 The construction of~\cite[Example~2.6]{Pmgm} produces, for any
field~$k$, a commutative $k$\+algebra $R$ with a principal ideal
$I=(s)$ for which the map $b_{I,R}\:\Delta_I(R)\rarrow\Lambda_I(R)$
is not an isomorphism.
 Hence $\Delta_I(R)$ is a nonquotseparated $I$\+contramodule
$R$\+module.

 Moreover, let $\R=\Lambda_I(R)$ be the $I$\+adic completion of
the same ring $R$, and let $\I=s\R=I\R=\Lambda_I(I)\subset\R$ be
the extension of the ideal $I\subset R$ in $\R$, or equivalently,
the $I$\+adic completion of the ideal~$I$.
 Then the $\R$\+module $\R$ is $\I$\+adically separated and complete;
so $\Delta_{\I}(\R)=\Lambda_{\I}(\R)=\R$.
 Still it follows from the discussion in Remark~\ref{main-remark}
below that, for any infinite set $X$, the map
$b_{\I,\R[X]}\:\Delta_{\I}(\R[X])\rarrow\Lambda_{\I}(\R[X])$ is not
an isomorphism (since all the $s$\+torsion in $\R$ is nondivisible,
and it is unbounded, as one can easily see).
 So $\Delta_{\I}(\R[X])$ is a nonquotseparated $\I$\+contramodule
$\R$\+module.

 The latter example also shows that the forgetful functor
$\R\contra\rarrow\R\modl_{\I\ctra}$ need not be an equivalence of
categories for a commutative ring $\R$ that is separated and
complete with respect to its finitely generated ideal~$\I$.
 In fact, Proposition~\ref{top-ring-contramodules} applied to
the ring $\R$ with the ideal $\I$ tells that $\R\contra\simeq
\R\modl_{\I\ctra}^\qs$; but $\R\modl_{\I\ctra}^\qs$ may still be
a proper full subcategory in $\R\modl_{\I\ctra}$.
 We refer to~\cite[Examples~5.2\,(7\+-8)]{Pper} for further
discussion.
\end{exs}

 Let us say a few words about category-theoretic properties of
the abelian categories of $I$\+contramodule $R$\+modules and
quotseparated $I$\+contramodule $R$\+modules.
 \emph{Neither} one of these abelian categories is Grothendieck.
 Indeed, let $I=(p)\subset\boZ=R$ be a maximal ideal in the ring of
integers, generated by a prime number~$p$; then the direct limit
of the sequence of monomorphisms $\boZ/p\boZ\rarrow\boZ/p^2\boZ
\rarrow\boZ/p^3\boZ\rarrow\dotsb$ vanishes in the abelian category
$\boZ_{p\ctra}=\boZ_{p\ctra}^\qs$.
 So does the direct limit of the sequence of monomorphisms
$\boZ_p\overset p\rarrow\boZ_p\overset p\rarrow\boZ_p\overset p\rarrow
\dotsb$ (where $\boZ_p$ denotes the abelian group of $p$\+adic
integers) \cite[Examples~4.4]{PR}.
 See~\cite{Neem} and the appendix to~\cite{Neem} for other examples of
similar behavior known in the literature.

 This also shows that contramodule categories have different nature
than the categories of ``condensed abelian groups'' or
``solid abelian groups''~\cite{SC}, where filtered direct limits
are exact~\cite[Theorems~2.2 and~5.8]{SC}.
 Infinite coproducts are exact in the category $\boZ_{p\ctra}$, because
it is a contramodule category of homological dimension~$1$
(see~\cite[Remark~1.2.1]{Pweak}), but \emph{not} in the abelian
categories $R\modl_{I\ctra}$ or $R\modl_{I\ctra}^\qs$ in general
(not even when the ring $R$ is Noetherian).

 On the other hand, both the abelian categories $R\modl_{I\ctra}$ and
$R\modl_{I\ctra}^\qs$ are always locally presentable.
 More precisely, they are locally $\aleph_1$\+presentable (see
book~\cite{AR} for the terminology).
 In the case of the category $R\modl_{I\ctra}$, this is explained
in~\cite[Example~4.1\,(3)]{PR} and~\cite[Examples~1.3\,(4)
and~2.2\,(1)]{Pper}.
 In the case of the category $R\modl_{I\ctra}^\qs$, one can use
Proposition~\ref{top-ring-contramodules} in order to reduce
the question to the general assertion about categories of
contramodules over topological rings~\cite[Sections~1.2 and~5]{PR},
\cite[Example~1.3\,(2)]{Pper}.

 Simply put, the abelian category $R\modl_{I\ctra}^\qs$ is the category
of modules over the monad $X\longmapsto\Lambda_I(R[X])$, and the abelian
category $R\modl_{I\ctra}$ is the category of modules over
the monad $X\longmapsto\Delta_I(R[X])$ on the category of sets.
 Both the abelian categories are locally $\aleph_1$\+presentable,
because both the functors $X\longmapsto\Lambda_I(R[X])$ and
$X\longmapsto\Delta_I(R[X])$ preserve countably-filtered direct limits.
 We refer to~\cite[Section~1.1]{PR}, \cite[Section~6]{PS1},
and~\cite[Section~1]{Pper} for a discussion of accessible additive
monads on the category of sets, which describe locally presentable
abelian categories with a projective generator.

\Section{Sequentially Derived Torsion and Complete Complexes}
\label{seq-secn}

 What Yekutieli~\cite{Yek3} calls ``derived complete complexes in
the sequential sense'' are best understood geometrically.

 Let us first introduce some notation and recall the definitions.
 Let $\s=(s_1,\dots,s_m)$ be a finite sequence of generators of
the ideal $I\subset R$ (see Section~\ref{introd-notation}).
 For every integer $n\ge1$, we denote by~$\s^n$ the sequence of
elements $s_1^n$,~\dots,~$s_m^n$.

 Let $K(R;\s)$ denote the Koszul complex
$$
 K(R;\s)=(R\overset{s_1}\rarrow R)\ot_R\dotsb\ot_R
  (R\overset{s_m}\rarrow R),
$$
which is a finite complex of finitely generated free $R$\+modules 
concentrated in the cohomological degrees~$-m$,~\dots,~$0$.
 The \emph{infinite dual Koszul complex}
$$
 K_\infty\spcheck(R;\s)=\varinjlim\nolimits_{n\ge1}\Hom_R(K(R;\s^n),R)
$$
is a finite complex of flat $R$\+modules concentrated in
the cohomological degrees~$0$,~\dots,~$m$.
 It can be computed as the tensor product
$$
 K_\infty\spcheck(R;\s)=(R\rarrow R[s_1^{-1}])\ot_R\dotsb\ot_R
 (R\rarrow R[s_m^{-1}]).
$$
 The \emph{\v Cech complex} $\check C(R;\s)$ is constructed as
the kernel of the natural surjective morphism of complexes
$K_\infty\spcheck(R;\s)\rarrow R$.

 There is an explicit construction of a bounded (($m+1$)\+term)
complex of countably generated free $R$\+modules quasi-isomorphic to
$K_\infty\spcheck(R;\s)$; it is called the ``telescope complex''
in~\cite[Section~5]{PSY} (see also~\cite[Section~2]{Pmgm}).

 One can show that the complexes $K_\infty\spcheck(R;\s)$ and
$\check C(R;\s)$ do not depend on the choice of a particular
sequence of generators of a given ideal $I\subset R$ up to
a natural isomorphism in $\sD(R\modl)$.
 In fact, the following lemma holds.
 
\begin{lem} \label{Koszul-Cech-determined-by-ideal}
 Let $I'=(s'_1,\dotsc,s'_m)$ and $I''=(s''_1,\dotsc,s''_n)
\subset R$ be two finitely generated ideals with the same radical
$\sqrt{I'}=\sqrt{I''}$.
 Put\/ $\s'=(s'_1,\dotsc,s'_m)$ and\/ $\s''=(s''_1,\dotsc,s''_n)$, and
let $(\s',\s'')$ denote the concatenation $(s'_1,\dotsc,s'_m,
\allowbreak s''_1,\dotsc,s''_n)$ of the two sequences.
 Then the natural morphisms of complexes of $R$\+modules
\begin{equation}  \label{infinite-dual-Koszul-comparison}
 K_\infty\spcheck(R;\s')\llarrow K_\infty\spcheck(R;(\s',\s''))\lrarrow
 K_\infty\spcheck(R;\s'')
\end{equation}
and
$$
 \check C(R;\s')\llarrow\check C(R;(\s',\s''))\lrarrow\check C(R;\s'')
$$
are quasi-isomorphisms.
\end{lem}

\begin{proof}
 This is~\cite[Corollary~6.2]{PSY} and~\cite[Proposition~2.20]{Yek3}.
 Alternatively, one can observe that all the three \v Cech complexes
in question compute the quasi-coherent sheaf cohomology $H^*(U,\cO_U)$
of the structure sheaf $\cO_U$ on the quasi-compact open complement
$U\subset\Spec R$ of the closed subset $Z\subset\Spec R$ corresponding
to the ideal $I'+I''$ (or which is the same, the closed subset
corresponding to $I'$ or~$I''$).
\end{proof}

 Given an element $s\in R$, an $R$\+module $M$ is said to be
\emph{$s$\+torsion} if for every $m\in M$ there exists $n\ge1$
such that $s^nm=0$ in~$M$.
 An $R$\+module $M$ is said to be \emph{$I$\+torsion} if it is
$s$\+torsion for every $s\in I$.
 We denote the Serre subcategory of $I$\+torsion $R$\+modules by
$R\modl_{I\tors}\subset R\modl$.

 A complex of $R$\+modules $M^\bu$ is said to be \emph{derived
$I$\+torsion in the sequential sense} \cite[Section~2]{Yek3} if
the complex $\check C(R;\s)\ot_R M^\bu$ is acyclic, or equivalently,
the canonical morphism of complexes $K_\infty\spcheck(R;\s)\ot_R
M^\bu\rarrow M^\bu$ is a quasi-isomorphism.

\begin{lem} \label{sequentially-torsion}
 The following three conditions are equivalent for a complex of
$R$\+modules~$M^\bu$: \hbadness=1050
\begin{enumerate}
\item $M^\bu$ is derived $I$\+torsion in the sequential sense;
\item the complex $R[s_j^{-1}]\ot_RM^\bu$ is acyclic for every
$j=1$,~\dots,~$m$;
\item the cohomology $R$\+module $H^n(M^\bu)$ is $I$\+torsion
for every $n\in\boZ$.
\end{enumerate}
\end{lem}

\begin{proof}
 (1)~$\Longrightarrow$~(2) The finite complex of flat $R$\+modules
$R[s_j^{-1}]\ot_R K_\infty\spcheck(R;\s)$ is contractible for every
$1\le j\le m$.
 Hence the complex $R[s_j^{-1}]\ot_RK_\infty\spcheck(R;\s)\ot_RM^\bu$
is acyclic for any complex of $R$\+modules~$M^\bu$.

 (2)~$\Longrightarrow$~(1) Every term of the finite complex of
$R$\+modules $\check C(R;\s)$ is a finite direct sum of $R$\+modules
of the form $R[(s_jt)^{-1}]\simeq R[s_j^{-1}]\ot_R R[t^{-1}]$ for
some $1\le j\le m$ and $t\in R$.
 Hence acyclicity of the complexes $R[s_j^{-1}]\ot_RM^\bu$ implies
acyclicity of the complex $\check C(R;\s)\ot_RM^\bu$.

 (2)~$\Longleftrightarrow$~(3) Clearly, an $R$\+module $N$ is
$I$\+torsion if and only if it is $s_j$\+torsion for every
$1\le j\le m$.
 Equivalently, the latter condition means that $R[s_j^{-1}]\ot_RN=0$.
 It remains to apply these observations to the $R$\+modules
$N=H^n(M^\bu)$ and recall that the $R$\+module $R[s_j^{-1}]$ is flat.
\end{proof}

 In other words, Lemma~\ref{sequentially-torsion} describes 
the category of derived $I$\+torsion complexes (of $R$\+modules)
in the sequential sense as the full subcategory $\sD_{I\tors}(R\modl)
\subset\sD(R\modl)$ of all complexes of $R$\+modules with
$I$\+torsion cohomology modules in the derived category of $R$\+modules.

 Dually, a complex of $R$\+modules $C^\bu$ is said to be
\emph{derived $I$\+adically complete in the sequential sense}
\cite[Section~2]{Yek3} if $\boR\Hom_R(\check C(R;\s),C^\bu)=0$, or
equivalently, the canonical morphism $C^\bu\rarrow
\boR\Hom_R(K_\infty\spcheck(R;\s),C^\bu)$
is an isomorphism in $\sD(R\modl)$.
 An equivalent definition can be found in~\cite[Definition tag~091S]{SP}:
a complex $C^\bu$ is ``derived complete with respect to~$I$'' if
$\boR\Hom_R(R[s^{-1}],C^\bu)=0$ for all $s\in I$.

\begin{lem} \label{sequentially-complete}
 The following three conditions are equivalent for a complex of
$R$\+modules~$C^\bu$: \hbadness=1050
\begin{enumerate}
\item $C^\bu$ is derived $I$\+adically complete in the sequential sense;
\item the object\/ $\boR\Hom_R(R[s_j^{-1}],C^\bu)\in\sD(R\modl)$
vanishes for all $j=1$,~\dots,~$m$;
\item the cohomology $R$\+module $H^n(C^\bu)$ is an $I$\+contramodule
for every $n\in\boZ$.
\end{enumerate}
\end{lem}

\begin{proof}
 The equivalence (1)~$\Longleftrightarrow$~(2) is provable in the same
way as in the previous lemma.
 The equivalence (2)~$\Longleftrightarrow$~(3) holds because every
$R$\+module $B$ satisfying $\Ext^*_R(R[s_j^{-1}],B)=0$ for all
$1\le j\le m$ is an $I$\+contramodule~\cite[Theorem~5.1]{Pcta} and
the projective dimension of the $R$\+modules $R[s^{-1}]$ does not
exceed~$1$.
\end{proof}

 Lemma~\ref{sequentially-complete} describes the category of derived
$I$\+adically complete complexes (of $R$\+modules) in the sequential
sense as the full subcategory $\sD_{I\ctra}(R\modl)\subset\sD(R\modl)$
of all complexes of $R$\+modules with $I$\+contramodule cohomology
modules in the derived category of $R$\+modules.

\medskip

 Now that we are finished with the definitions and basic lemmas, we can
have the geometric discussion promised in the beginning of this section.
 The basic concepts are the \emph{quasi-coherent sheaves} and
the \emph{contraherent cosheaves} on a scheme~$X$.
 The former is well-known, and the latter was introduced in
the preprint~\cite{Pcosh}.

 The definition of a contraherent cosheaf is obtained by dualizing
a suitably formulated definition of a quasi-coherent sheaf.
 Let us recall the Enochs--Estrada intepretation of quasi-coherent
sheaves~\cite[Section~2]{EE}, which is convenient for dualization.
 A quasi-coherent sheaf $\F$ on a scheme $X$ is the same thing as
a rule assigning to every affine open subscheme $U\subset X$
an $\cO(U)$\+module $\F(U)$ and to every pair of affine open subschemes
$V\subset U\subset X$ a morphism of $\cO(U)$\+modules $\F(U)\rarrow
\F(V)$ in such a way that the induced morphism of $\cO(V)$\+modules
$\cO(V)\ot_{\cO(U)}\F(U)\rarrow\F(V)$ is an isomorphism, and
the triangle diagrams $\F(U)\rarrow\F(V)\rarrow\F(W)$ are commutative
for all triples of affine open subschemes $W\subset V\subset U
\subset X$.
 One can check that any such set of data extends uniquely to
a quasi-coherent sheaf of $\cO_X$\+modules defined on all (and not
only affine) open subsets in~$X$.

 Dually, a \emph{contraherent cosheaf} $\fP$ on a scheme $X$ is a rule
assigning to every affine open subscheme $U\subset X$
an $\cO(U)$\+module $\fP[U]$ and to every pair of affine open
subschemes $V\subset U\subset X$ a morphism of $\cO(U)$\+modules
$\fP[V]\rarrow\fP[U]$ such that the induced morphism of
$\cO(V)$\+modules $\fP[V]\rarrow\Hom_{\cO(U)}(\cO(V),\fP[U])$ is
an isomorphism, $\Ext^1_{\cO(U)}(\cO(V),\fP[U])=0$, and
the triangle diagrams $\fP[W]\rarrow\fP[V]\rarrow\fP[U]$ are
commutative.
 Here the Ext vanishing condition only needs to be imposed for
$\Ext^1$, as the projective dimension of the $\cO(U)$\+module
$\cO(V)$ never exceeds~$1$ \,\cite[beginning of Section~1.1 and
Lemma~1.2.4]{Pcosh} (no such condition was needed in the quasi-coherent
context, as the $\cO(U)$\+module $\cO(V)$ is flat).
 Furthermore, any such set of data extends uniquely to a cosheaf
of $\cO_X$\+modules defined on all (and not only affine) open
subsets in~$X$ \,\cite[Theorem~2.2.1]{Pcosh}.

 We will not reproduce here the constructions of the direct and inverse
image functors of contraherent cosheaves (denoted by $f_!$ and $f^!$,
respectively, for a nice enough scheme morphism~$f$).
 These are dual to the familiar constructions for quasi-coherent
sheaves, and their properties are dual.
 We refer to~\cite[Section~2.3]{Pcosh} for the details.

 For any quasi-compact semi-separated scheme $Y$, the derived category
of the abelian category of quasi-coherent sheaves on $Y$ is equivalent
to the derived category of the exact category of contraherent cosheaves
on~$Y$ \,\cite[Theorem~4.6.6]{Pcosh}
$$
 \sD^\star(Y\qcoh)\simeq\sD^\star(Y\ctrh)
$$
for every bounded or unbounded derived category symbol $\star=\bs$,
$+$, $-$, or~$\varnothing$.
 Furthermore, for any morphism of quasi-compact semi-separated
schemes $f\:Y\rarrow X$, the equivalences of categories
$\sD^\star(Y\qcoh)\simeq\sD^\star(Y\ctrh)$ and $\sD^\star(X\qcoh)
\simeq\sD^\star(X\ctrh)$ transform the right derived direct image
functor of quasi-coherent sheaves
$$
 \boR f_*\:\sD^\star(Y\qcoh)\lrarrow\sD^\star(X\qcoh)
$$
into the left derived direct image functor of contraherent cosheaves
$$
 \boL f_!\:\sD^\star(Y\ctrh)\lrarrow\sD^\star(X\ctrh),
$$
so $\boR f_*=\boL f_!$ \cite[Theorem~4.8.1]{Pcosh}.

 In particular, for an affine scheme $X=\Spec R$, the abelian category
$X\qcoh$ is equivalent to the category of $R$\+modules.
 The exact category $X\ctrh$ is a full subcategory in $R\modl$
consisting of all the \emph{contraadjusted} $R$\+modules (that is,
$R$\+modules that are $s$\+contraadjusted for all $s\in R$, in
the sense of the definition in
Section~\ref{derived-complete-modules-secn}).
 This restriction does not affect the derived category:
the inclusion functor $X\ctrh\rarrow R\modl$ induces an equivalence
$\sD^\star(X\ctrh)\simeq\sD^\star(R\modl)$.

 Let $Z\subset X$ be the closed subscheme $Z=\Spec R/I$, and
let $U=X\setminus Z$ be its open complement.
 Let $j\:U\rarrow X$ denote the open embedding morphism.
 Then the functor $\boR j_*$ has a left adjoint functor
$$
 j^*\:\sD^\star(X\qcoh)\lrarrow\sD^\star(U\qcoh),
$$
while the functor $\boL j_!$ has a right adjoint functor
$$
 j^!\:\sD^\star(X\ctrh)\lrarrow\sD^\star(U\ctrh).
$$
 Moreover, both the compositions $j^*\circ\boR j_*$ and
$j^!\circ\boL j_!$ are identity endofunctors.
 Hence one obtains a \emph{recollement}, which is described purely
algebraically in~\cite[Section~3]{Pmgm} (in some form, these
results go back to~\cite[Section~6]{DG}).

 Then the full triangulated subcategory of derived $I$\+torsion complexes
in the sequential sense $\sD^\star_{I\tors}(R\modl)\subset\sD(R\modl)$
can be described as the kernel of the quasi-coherent open restriction
functor~$j^*$,
$$
 \sD^\star_{I\tors}(R\modl)\,=\,
 \ker(j^*\:\sD^\star(X\qcoh)\rarrow\sD^\star(U\qcoh)).
$$
 Dually, the full triangulated  subcategory of derived $I$\+adically
complete complexes in the sequential sense $\sD^\star_{I\ctra}(R\modl)
\subset\sD^\star(R\modl)$ can be described as the kernel of contraherent
open restriction functor~$j^!$,
$$
 \sD^\star_{I\ctra}(R\modl)\,=\,
 \ker(j^!\:\sD^\star(X\ctrh)\rarrow\sD^\star(U\ctrh)).
$$

 In fact, it follows from the existence of the recollement that
the categories of complexes with $I$\+torsion and with $I$\+contramodule
cohomology modules are equivalent \cite[Proposition~3.3
and Theorem~3.4]{Pmgm}
\begin{equation} \label{general-mgm-equivalence}
 \sD^\star_{I\tors}(R\modl)\simeq
 \sD^\star(R\modl)/\sD^\star(U)\simeq
 \sD^\star_{I\ctra}(R\modl).
\end{equation}
 Here $\sD^\star(U)$ is a notation for the category $\sD^\star(U\qcoh)
\simeq\sD^\star(U\ctrh)$, embedded into $\sD(R\modl)$ by the functor
$\boR j_*=\boL j_!$.
 The two equivalences in~\eqref{general-mgm-equivalence} are provided by
the compositions of the identity inclusions $\sD^\star_{I\tors}(R\modl)
\rarrow\sD^\star(R\modl)$ and $\sD^\star_{I\ctra}(R\modl)
\rarrow\sD^\star(R\modl)$ with the Verdier quotient functor
$\sD^\star(R\modl)\rarrow\sD^\star(R\modl)/\sD^\star(U)$.
 So the recollement takes the form
\begin{equation} \label{general-mgm-recollement}
\begin{gathered}
 \xymatrix{
   \sD^\star(U\qcoh) \ar@{=}[d] &
   \sD^\star(X\qcoh) \ar@{->>}[l]_{{\ j^*}} \ar@{=}[d] &
   \sD^\star_{I\tors}(R\modl) \ar@{=}[d] \ar@{>->}[l] \\
   \sD^\star(U) \ar@{=}[d] \ar@{>->}[r]^{{\boR j_*=\boL j_!\ \ \ }} &
   \sD^\star(R\modl) \ar@{=}[d] \ar@{->>}[r] &
   \sD^\star(R\modl)/\sD^\star(U) \ar@{=}[d] \\
   \sD^{\star}(U\ctrh) & \sD^\star(X\ctrh) \ar@{->>}[l]_{{\ j^!}} &
   \sD^\star_{I\ctra}(R\modl) \ar@{>->}[l]
 }
\end{gathered}
\end{equation}
where two-headed arrows $\twoheadrightarrow$ denote triangulated Verdier
quotient functors and arrows with a tail $\rightarrowtail$ denote
triangulated fully faithful embeddings.
 After the identifications in the vertical equation signs (explained in
the discussion above), the functors in the upper row are left adjoint to
those in the middle row, which are left adjoint to those in
the lower row.
 In every row, the image of the fully faithful embedding is equal to
the kernel of the Verdier quotient functor.

 The right adjoint functor to the inclusion (the coreflector)
$\sD^\star(R\modl)\rarrow\sD_{I\tors}^\star(R\modl)$ is computable
algebraically as the triangulated functor $M^\bu\longmapsto
K_\infty\spcheck(R;\s)\ot_RM^\bu$.
 This is called the \emph{sequential derived $I$\+torsion functor}
in~\cite{Yek3}.
 The sequential derived $I$\+torsion functor is idempotent because
it is a coreflector.
 One can also see it directly from the quasi-isomorphism
\begin{equation} \label{Koszul-tensor-idempotent}
 K_\infty\spcheck(R;\s)\ot_RK_\infty\spcheck(R;\s) \lrarrow
 K_\infty\spcheck(R;\s),
\end{equation}
which is a particular case of
Lemma~\ref{Koszul-Cech-determined-by-ideal}.
 In fact, there are two quasi-isomorphisms of complexes
in~\eqref{Koszul-tensor-idempotent}, given by the two arrows
in~\eqref{infinite-dual-Koszul-comparison} for $\s'=\s=\s''$.
 The two maps of complexes induce the same isomorphism
in $\sD(R\modl)$.

 Dually, the left adjoint functor to the inclusion (the reflector)
$\sD^\star(R\modl)\rarrow\sD_{I\ctra}^\star(R\modl)$ is computable
algebraically as the triangulated functor $C^\bu\longmapsto
\boR\Hom_R(K_\infty\spcheck(R;\s),\>\allowbreak C^\bu)$.
 This is called the \emph{sequential derived $I$\+adic completion
functor} in~\cite{Yek3}.
 The sequential derived $I$\+adic completion functor is idempotent
because it is a reflector; one can also see it directly from
the quasi-isomorphism~\eqref{Koszul-tensor-idempotent}
(cf.~\cite[Section~1 and Remark~2.23]{Yek3}).

 The restrictions of the functors
$\boR\Hom_R(K_\infty\spcheck(R;\s),{-})$ and
$K_\infty\spcheck(R;\s)\ot_R{-}$ onto the full subcategories
$\sD_{I\tors}^\star(R\modl)$ and $\sD_{I\ctra}^\star(R\modl)
\subset\sD^\star(R\modl)$ provide the mutually inverse
equivalences $\sD_{I\tors}^\star(R\modl)\simeq
\sD_{I\ctra}^\star(R\modl)$.

 Somewhat similarly, the reflector $j^*\:\sD(X\qcoh)\rarrow\sD(U\qcoh)$
onto the full triangulated subcategory
$j_*(\sD(U\qcoh))\subset\sD(X\qcoh)$ is computable algebraically as
the functor $M^\bu\longmapsto\check C(R;\s)\ot_R M^\bu$.
 The coreflector $j^!\:\sD(X\ctrh)\rarrow\sD(U\ctrh)$ onto
(the same, under the identification $\sD(X\qcoh)\simeq\sD(X\ctrh)$)
full triangulated subcategory $j_!(\sD(U\ctrh))\subset\sD(X\ctrh)$ is
computable algebraically as the functor
$C^\bu\longmapsto\boR\Hom_R(\check C(R;\s),\>C^\bu)$.

\begin{rem}
 The reader should be warned that the $!$\+notation (as in $f_!$
and~$f^!$, or $j_!$ and~$j^!$) in the above exposition, as well as
generally in~\cite{Pcosh}, refers to what is called
\emph{Neeman's extraordinary inverse image functor} in~\cite{Pcosh},
with the reference to~\cite{Neem0}.
 It should be distinguished from \emph{Deligne's extraordinary
inverse image functor} constructed by Hartshorne in~\cite{Hart}
and by Deligne in the appendix to~\cite{Hart}.
 In particular, the $!$\+notation in~\cite[Lecture~XI]{SC} stands for
an extraordinary inverse image functor in the sense of Deligne.
 We refer to~\cite[Appendix]{Hart} or~\cite[Introduction and
Section~5.16]{Pcosh} for a discussion.
\end{rem}

\Section{Weak Proregularity}  \label{wpr-secn}

 Let $C$ be an $R$\+module.
 The object $\boR\Hom_R(K_\infty\spcheck(R;\s),\>C)$ of the derived
category of $R$\+modules $\sD(R\modl)$ plays an important role in
our considerations.
 Its cohomology modules appear in the natural short exact sequences
of $R$\+modules
\begin{multline} \label{derived-projective-limit-sequence}
 0\lrarrow\varprojlim\nolimits_{n\ge1}^1 H^{q-1}(K(R;\s^n)\ot_R C)
 \\ \lrarrow H^q\,\boR\Hom_R(K_\infty\spcheck(R;\s),\>C) \\ \lrarrow
 \varprojlim\nolimits_{n\ge1} H^q(K(R;\s^n)\ot_RC)\lrarrow0,
\end{multline}
where $\varprojlim_{n\ge1}^1$ denotes the first derived countable
projective limit.
 The sequence~\eqref{derived-projective-limit-sequence} may be
nontrivial for integers~$q$ in the interval $-m\le q\le 0$.
 (We recall that, following the notation in
Section~\ref{introd-notation}, \,$m$~is the length of the sequence
of generators $\s=(s_1,\dots,s_m)$ of the ideal $I\subset R$.)

 In particular, the $R$\+module $\Delta_I(C)$ is
computable~\cite[Theorem~7.2(iii)]{Pcta} as
\begin{equation} \label{delta-computed}
 \Delta_I(C)=H^0\,\boR\Hom_R(K_\infty\spcheck(R;\s),\>C),
\end{equation}
while the $I$\+adic completion of the $R$\+module $C$ is
\begin{equation} \label{lambda-computed}
 \Lambda_I(C)=\varprojlim\nolimits_{n\ge1} H^0(K(R;\s^n)\ot_RC).
\end{equation}

 A countable projective system of abelian groups $(E_n)_{n\ge1}$ is said
to satisfy the \emph{Mittag-Leffler condition} if for every  $i\ge 1$
there exists $j\ge i$ such that the images of the transition maps
$E_k\rarrow E_i$ coincide (as subgroups in~$E_i$) for all $k\ge j$.
 A projective system $(E_n)_{n\ge1}$ is said to be \emph{pro-zero}
(an alternative terminology is that $(E_n)_{n\ge1}$ satisfies
the \emph{trivial Mittag-Leffler condition}) if for every $i\ge1$ there
exists $j\ge i$ such that the transition map $E_j\rarrow E_i$ is zero.

\begin{lem} \label{Mittag-Leffler+projlim-vanishing}
 A countable projective system of abelian groups $(E_n)_{n\ge1}$ is
pro-zero if and only if it the following two conditions hold:
\begin{enumerate}
\renewcommand{\theenumi}{\roman{enumi}}
\item the projective system $(E_n)_{n\ge1}$ satisfies
the Mittag-Leffler condition; and
\item $\varprojlim_{n\ge1}E_n=0$.
\end{enumerate}
\end{lem}

\begin{proof}
 The ``only if'' implication is obvious.
 To prove the ``if'', suppose that a projective system $(E_n)_{n\ge1}$
satisfies the Mittag-Leffler condition, and, for every $i\ge1$,
denote by $E'_i\subset E_i$ the image of the transition map $E_j
\rarrow E_i$ for $j$~large enough.
 Then $(E'_n)_{n\ge1}$ is a projective system of surjective maps of
abelian groups such that $\varprojlim_{n\ge1}E'_n=
\varprojlim_{n\ge1}E_n$.
 Hence the projection maps $\varprojlim_{n\ge1}E'_n\rarrow E'_i$ are
surjective, and $\varprojlim_{n\ge1}E_n=0$ implies $E'_i=0$ for all
$i\ge1$.
\end{proof}

 For any countable projective system of abelian groups $(E_n)_{n\ge1}$
satisfying the Mittag-Leffler condition, one has
$\varprojlim_{n\ge1}^1E_n=0$.
 The following lemma provides a converse implication.

\begin{lem} \label{emmanouil}
 For any countable projective system of abelian groups $(E_n)_{n\ge1}$,
the following two conditions are equivalent:
\begin{enumerate}
\item the projective system $(E_n)_{n\ge1}$ satisfies
the Mittag-Leffler condition;
\item $\varprojlim_{n\ge1}^1 E_n^{(X)}=0$ for some (equivalently,
every) infinite set~$X$.
\end{enumerate}
 Here we denote by $E^{(X)}$ the $X$\+indexed direct sum of copies of
an abelian group~$E$.
\end{lem}

\begin{proof}
 This is the result of the paper~\cite{Emm}
(see~\cite[Corollary~6\,(i)\,$\Leftrightarrow$\,(iii)]{Emm}).
\end{proof}

 We recall the following definition, which plays a key role, and refer
to~\cite[Section~3]{Yek3} for a discussion of its history.
 The ideal $I\subset R$ is said to be \emph{weakly proregular} if,
for every fixed $q<0$, the projective system
$$
 (H^q(K(R;\s^n)))_{n\ge1}
$$
is pro-zero.
 This property does not depend on the choice of a particular finite
system of generators $s_1$,~\dots, $s_m$ of the ideal~$I$
\cite[Corollary~6.3]{PSY}.
 Furthermore, the following assertion holds.

\begin{thm} \label{weak-proregularity-via-injectives}
 The ideal $I\subset R$ is weakly proregular if and only if, for
every injective $R$\+module $J$ and every integer $k>0$, one has
$H^k(K_\infty\spcheck(R;\s)\ot_RJ)=0$.
\end{thm}

\begin{proof}
 This result goes back to Grothendieck~\cite[Lemma~2.4]{LC};
see also~\cite[Theorem~3.2]{Sch} or~\cite[Theorem~4.24]{PSY}.
\end{proof}

 For any $R$\+module $M$, we denote by $\Gamma_I(M)\subset M$
the maximal $I$\+torsion submodule in~$M$.
 The functor $\Gamma_I\:R\modl\rarrow R\modl_{I\tors}$ is right
adjoint to the exact, fully faithful inclusion functor
$R\modl_{I\tors}\rarrow R\modl$.

\begin{cor} \label{when-weakly-proregular-in-terms-of-injectives}
 The ideal $I\subset R$ is weakly proregular if and only if for
every injective $R$\+module $J$ the canonical morphism of complexes
of $R$\+modules
$$
 \Gamma_I(J)\lrarrow K_\infty\spcheck(R;\s)\ot_R J
$$
is a quasi-isomorphism.
\end{cor}

\begin{proof}
 In fact, for any $R$\+module $M$ there is a natural isomorphism of
$R$\+modules $\Gamma_I(M)\simeq H^0(K_\infty\spcheck(R;\s)\ot_R M)$.
 Hence the corollary follows from
Theorem~\ref{weak-proregularity-via-injectives}.
\end{proof}

 The next proposition and theorem, providing dual versions of
Theorem~\ref{weak-proregularity-via-injectives} and
Corollary~\ref{when-weakly-proregular-in-terms-of-injectives},
are the main results of this section.

\begin{prop} \label{weak-proregularity-piece-by-piece}
 The ideal $I\subset R$ is weakly proregular if and only if, for
every $q<0$, the following two conditions hold:
\begin{enumerate}
\renewcommand{\theenumi}{\roman{enumi}}
\item $\varprojlim_{n\ge1}^1 H^q(K(R;\s^n)\ot_R R[X])=0$ for some
(equivalently, every) infinite set~$X$;
\item $\varprojlim_{n\ge1} H^q(K(R;\s^n)\ot_R R[X])=0$ for some
(equivalently, every) nonempty set~$X$.
\end{enumerate}
 Here $R[X]=R^{(X)}$ denotes the free $R$\+module with $X$~generators.
\end{prop}

\begin{proof}
 Clearly, for any projective system $(E_n)_{n\ge1}$ and a nonempty
set $X$ one has $\varprojlim_{n\ge1}E_n=0$ if and only if
$\varprojlim_{n\ge1}E_n^{(X)}=0$.
 So the conditions in~(ii) are equivalent for all nonempty sets~$X$,
and they are equivalent to the condition~(ii) of
Lemma~\ref{Mittag-Leffler+projlim-vanishing} for the projective system
$E_n=H^q(K(R;\s^n))$.
 Furthermore, by Lemma~\ref{emmanouil}, the conditions in~(i) are
equivalent for all infinite sets~$X$, and they are equivalent to
the condition~(i) of Lemma~\ref{Mittag-Leffler+projlim-vanishing} for
the same projective system~$(E_n=H^q(K(R;\s^n)))_{n\ge1}$.
 The assertion of Lemma~\ref{Mittag-Leffler+projlim-vanishing} now
provides the desired result.
\end{proof}

 The following theorem can be also found in~\cite[Remark~7.8]{HS}.

\begin{thm} \label{when-weakly-proregular}
 The ideal $I\subset R$ is weakly proregular if and only if for some
(equivalently, for every) infinite set $X$, the canonical morphism
$$
 \boR\Hom_R(K_\infty\spcheck(R;\s),\>R[X])\lrarrow\Lambda_I(R[X])
$$
is an isomorphism in the derived category of $R$\+modules\/
$\sD(R\modl)$.
\end{thm}

\begin{proof}
 Follows immediately from the short exact
sequences~\eqref{derived-projective-limit-sequence}
together with the isomorphism~\eqref{lambda-computed}
and Proposition~\ref{weak-proregularity-piece-by-piece}.
\end{proof}

\begin{cor} \label{all-quotseparated-over-weakly-proregular}
 If the ideal $I\subset R$ is weakly proregular, then
the full subcategories $R\modl_{I\ctra}^\qs$ and $R\modl_{I\ctra}$
coincide in $R\modl$.
\end{cor}

\begin{proof}
 Follows from Lemma~\ref{when-all-quotseparated},
the isomorphism~\eqref{delta-computed}, and
Theorem~\ref{when-weakly-proregular}.
\end{proof}

\begin{rem} \label{main-remark}
 The converse assertion to
Corollary~\ref{all-quotseparated-over-weakly-proregular} is
\emph{not} true: the condition that the two full subcategories
$R\modl_{I\ctra}^\qs$ and $R\modl_{I\ctra}$ in $R\modl$ coincide is
weaker than the weak proregularity of the ideal~$I$.
 In fact, Proposition~\ref{weak-proregularity-piece-by-piece}
splits the weak proregularity condition for a sequence of elements
$s_1$,~\dots, $s_m\in R$ into $2m$~pieces.
 Precisely one of these $2m$~conditions, namely, condition~(i)
for $q=-1$, is equivalent to every $I$\+contramodule $R$\+module
being quotseparated.

 So, what do the remaining $2m-1$ conditions do?
 As we will see in Section~\ref{full-and-faithfulness-secn},
the answer is that condition~(ii) for all $q\le -1$ and condition~(i)
for all $q\le -2$ hold together if and only if the triangulated functor
$\sD(R\modl_{I\ctra})\rarrow\sD(R\modl)$ induced by the inclusion of
abelian categories $R\modl_{I\ctra}\rarrow R\modl$ is fully faithful.
 It follows that the ideal $I\subset R$ is weakly proregular (i.~e.,
all the $2m$~conditions hold) if and only if the triangulated functor
$\sD(R\modl_{I\ctra}^\qs)\rarrow\sD(R\modl)$ induced by the inclusion
of abelian categories $R\modl_{I\ctra}^\qs\rarrow R\modl$
is fully faithful.

 The special case of $m=1$ is instructive.
 A principal ideal $I=(s)$ is weakly proregular if and only if
the $s$\+torsion in the ring $R$ is bounded, that is, there exists
an integer $n_0\ge1$ such that for any integer $n\ge1$ and any
element $r\in R$ the equation $s^nr=0$ implies $s^{n_0}r=0$.
 This condition splits into two conditions~(i) and~(ii) of
Proposition~\ref{weak-proregularity-piece-by-piece} in the following
way.

 Given an $R$\+module $M$, the \emph{$s$\+torsion submodule} of $M$
consists of all the elements $m\in M$ for which there exists $n\ge1$
such that $s^nm=0$ in $M$.
 We say that an $R$\+module $M$ has \emph{no divisible $s$\+torsion} if
the module $\Hom_R(R[s^{-1}]/R,\>M)$ vanishes
(where $R[s^{-1}]/R$ is a simplified notation for the cokernel of
the localization map $R\rarrow R[s^{-1}]$).
 The \emph{divisible $s$\+torsion submodule} of $M$ consists of all
elements belonging to (the union of) the images of $R$\+module maps
$R[s^{-1}]/R\rarrow M$.
 Finally, we say that \emph{the nondivisible $s$\+torsion in $M$ is
bounded} if there exists $n_0\ge1$ such that the $s$\+torsion
submodule in $M$ is the sum of the divisible $s$\+torsion submodule
and the kernel of the map $s^{n_0}\:M\rarrow M$.

 Then condition~(i) means that the nondivisible $s$\+torsion in
the $R$\+module $R$ is bounded, while condition~(ii) is equivalent
to $R$ having no divisible $s$\+torsion.
 So one has $R\modl_{s\ctra}=R\modl_{s\ctra}^\qs$ if and only if
the nondivisible $s$\+torsion in $R$ is bounded, while the functor
$\sD(R\modl_{s\ctra})\rarrow\sD(R\modl)$ is fully faithful if and
only if $R$ has no divisible $s$\+torsion.
 The functor $\sD(R\modl_{I\ctra}^\qs)\rarrow\sD(R\modl)$ is fully
faithful if and only if $R$ has bounded $s$\+torsion.
\end{rem}

\Section{Full-and-Faithfulness of Triangulated Functors \\
and Weak Proregularity}  \label{full-and-faithfulness-secn}

 We prove several basic theorems in this section before proceeding
to discuss derived $I$\+adically complete complexes in the idealistic
sense in the next one.

 To begin with, here is the $I$\+torsion version of the main theorem.
 Let us introduce some notation.
 Let $\star=\bs$, $+$, $-$, or~$\varnothing$ be a bounded or unbounded
derived category symbol.
 Denote the triangulated functor $\sD^\star(R\modl_{I\tors})\rarrow
\sD^\star(R\modl)$ induced by the inclusion of abelian categories
$R\modl_{I\tors}\rarrow R\modl$ by
$$
 \mu^\star\:\sD^\star(R\modl_{I\tors})\lrarrow
 \sD^\star(R\modl).
$$
 The functor~$\mu^\star$ obviously factorizes into a composition
\begin{equation} \label{mu-chi-upsilon-factorization}
 \sD^\star(R\modl_{I\tors})\overset{\chi^\star}\lrarrow
 \sD_{I\tors}^\star(R\modl)\overset{\upsilon^\star}\lrarrow
 \sD(R\modl),
\end{equation}
where $\upsilon^\star$~is the canonical fully faithful inclusion.

 In the proofs of the theorems in this section, we will use the concept
of a partially defined adjoint functor.
 Given two categories $\sC$, $\sD$ and a functor $F\:\sC
\rarrow\sD$, we say that a \emph{partial} functor $G$ right adjoint
to $F$ \emph{is defined} on an object $D\in\sD$ if there exists
an object $G(D)\in\sC$ such that for every object $C\in\sC$ there is
a bijection of sets $\Hom_\sC(C,G(D))\simeq\Hom_\sD(F(C),D)$
functorial in the object $C\in\sC$.
 Since an object representing a functor is unique up to a unique
isomorphism, the object $G(D)$ is unique if it exists.
 Partial left adjoint functors are defined similarly.

\begin{thm} \label{I-torsion-theorem}
 If the ideal $I\subset R$ is weakly proregular, then for every
symbol\/ $\star=\bs$, $+$, $-$, or\/~$\varnothing$, the functor
$\chi^\star\:\sD^\star(R\modl_{I\tors})\rarrow
\sD_{I\tors}^\star(R\modl)$ is an equivalence of categories and
the functor $\mu^\star\:\sD^\star(R\modl_{I\tors})\rarrow
\sD^\star(R\modl)$ is fully faithful.

 Conversely, if for one of the symbols\/ $\star=\bs$, $+$, $-$,
or\/~$\varnothing$ the functor $\mu^\star$ is fully faithful, then
the ideal $I\subset R$ is weakly proregular.
\end{thm}

\begin{proof}
 The direct assertion is the result of~\cite[Theorem~1.3 and
Corollary~1.4]{Pmgm}.
 Let us prove the converse.
 Clearly, if the functor $\mu^\star$ is fully faithful for one of
the symbols $\star=\bs$, $+$, $-$, or~$\varnothing$, then it is
fully faithful for $\star=\bs$.
 It is also clear that if the functor~$\mu^\bs$ is fully faithful,
then its essential image coincides with the full subcategory
$\sD^\bs_{I\tors}(R\modl)\subset\sD^\bs(R\modl)$; so
the functor~$\chi^\bs$ is a category equivalence in this case.

 Now the functor of inclusion of abelian categories $R\modl_{I\tors}
\rarrow R\modl$ has a right adjoint functor $\Gamma_I\:
R\modl\rarrow R\modl_{I\tors}$.
 Being right adjoint to an exact functor, the functor $\Gamma_I$
takes injective $R$\+modules to injective objects of the category
of $I$\+torsion $R$\+modules.
 One easily concludes that the functor $\mu^+\:\sD^+(R\modl_{I\tors})
\rarrow\sD^+(R\modl)$ has a right adjoint functor
$\gamma^+\:\sD^+(R\modl)\rarrow\sD^+(R\modl_{I\tors})$, which is
computable as the right derived functor of the functor~$\Gamma_I$.
 So the functor~$\gamma^+$ assigns to a bounded below complex of
injective $R$\+modules $J^\bu$ the bounded below complex of
injective $I$\+torsion $R$\+modules $\Gamma_I(J^\bu)$.
 Similarly, for $\star=\varnothing$, the functor
$\mu\:\sD(R\modl_{I\tors})\rarrow\sD(R\modl)$ has a right adjoint
functor $\gamma\:\sD(R\modl)\rarrow\sD(R\modl_{I\tors})$, which can be
computed by applying the functor $\Gamma_I$ to homotopy injective
complexes of $R$\+modules (known also as ``K\+injective
complexes''~\cite{Spa}).

 For $\star=\bs$, the functor $\mu^\bs\:\sD^\bs(R\modl_{I\tors})\rarrow
\sD^\bs(R\modl)$ does not seem to necessarily have a right adjoint, in
general; but we are interested in its partially defined right adjoint
functor~$\gamma^\bs$.
 It is only important for us that the partial functor~$\gamma^\bs$ is
defined on injective $R$\+modules $J\in R\modl_\inj\subset R\modl
\subset\sD^\bs(R\modl)$; indeed, one has $\gamma^\bs(J)=\Gamma_I(J)\in
R\modl_{I\tors}\subset\sD^\bs(R\modl_{I\tors})$, as can be easily seen.

 On the other hand, for any symbol $\star=\bs$, $+$, $-$,
or~$\varnothing$, the fully faithful inclusion functor
$\upsilon^\star\:\sD^\star_{I\tors}(R\modl)\rarrow
\sD^\star(R\modl)$ has a right adjoint functor
$\theta^\star\:\sD^\star(R\modl)\rarrow\sD^\star_{I\tors}(R\modl)$.
 This holds quite generally for any finitely generated ideal
$I\subset R$; following the discussion in Section~\ref{seq-secn}
and~\cite[Section~3]{Pmgm}, the functor~$\theta^\star$ is computable as
$$
 \theta^\star(M^\bu)=K_\infty\spcheck(R;\s)\ot_R M^\bu
 \quad\text{for all $M^\bu\in\sD^\star(R\modl)$}.
$$

 Finally, if the functor $\chi^\bs\:\sD^\bs(R\modl_{I\tors})\rarrow
\sD^\bs_{I\tors}(R\modl)$ is a category equivalence, then it
identifies the functor~$\mu^\bs$ with the functor~$\upsilon^\bs$;
and consequently it also identifies their adjoint functors
$\gamma^\bs$ and~$\theta^\bs$.
 It follows that the functor~$\gamma^\bs$ is everywhere defined in
this case; but this is not important for us.
 The key observation is that the objects
$\gamma^\bs(J)\in\sD^\bs(R\modl_{I\tors})$ and
$\theta^\bs(J)\in\sD^\bs_{I\tors}(R\modl)$ are identified by
the functor~$\chi^\bs$, for any injective $R$\+module~$J$.
 In other words, we have an isomorphism of complexes of $R$\+modules
$$
 \Gamma_I(J)\simeq K_\infty\spcheck(R;\s)\ot_RJ.
$$

 Hence $H^k(K_\infty\spcheck(R;\s)\ot_RJ)=0$ for all $k>0$.
 According to Theorem~\ref{weak-proregularity-via-injectives},
it follows that the ideal $I\subset R$ is weakly proregular.
\end{proof}

 The following quotseparated $I$\+contramodule theorem is the main
result of this section.
 In order to formulate and prove it, we need some further notation.
 Denote the triangulated functor $\sD^\star(R\modl_{I\ctra}^\qs)
\rarrow\sD^\star(R\modl)$ induced by the inclusion of the abelian
categories $R\modl_{I\ctra}^\qs\rarrow R\modl$ by
$$
 \rho^\star\:\sD^\star(R\modl_{I\ctra}^\qs)\lrarrow
 \sD^\star(R\modl).
$$
 The functor~$\rho^\star$ obviously factorizes into a composition
\begin{equation} \label{rho-xi-iota-factorization}
 \sD^\star(R\modl_{I\ctra}^\qs)\overset{\xi^\star}\lrarrow
 \sD^\star_{I\ctra}(R\modl)\overset{\iota^\star}\lrarrow
 \sD^\star(R\modl),
\end{equation}
where $\iota^\star$~is the canonical fully faithful inclusion.

\begin{thm} \label{quotseparated-I-contramodule-theorem}
 If the ideal $I\subset R$ is weakly proregular, then for every
symbol\/ $\star=\bs$, $+$, $-$, or\/~$\varnothing$, the functor
$\xi^\star\:\sD^\star(R\modl_{I\ctra}^\qs)\rarrow
\sD^\star_{I\ctra}(R\modl)$ is an equivalence of categories and
the functor $\rho^\star\:\sD^\star(R\modl_{I\ctra}^\qs)\rarrow
\sD^\star(R\modl)$ is fully faithful.

 Conversely, if for one of the symbols\/ $\star=\bs$, $+$, $-$,
or\/~$\varnothing$ the functor~$\rho^\star$ is fully faithful,
then the ideal $I\subset R$ is weakly proregular.
\end{thm}

\begin{proof}
 If the ideal $I$ is weakly proregular, then we have
$R\modl_{I\ctra}^\qs=R\modl_{I\ctra}$ by
Corollary~\ref{all-quotseparated-over-weakly-proregular}.
 Having made this observation, it remains to refer
to~\cite[Theorem~2.9 and Corollary~2.10]{Pmgm} for the proof of
the direct assertion.
 
 To prove the converse, we argue similarly (or rather,
dual-analogously) to the proof of Theorem~\ref{I-torsion-theorem}.
 Clearly, if the functor~$\rho^\star$ is fully faithful for one of
the symbols $\star=\bs$, $+$, $-$, or~$\varnothing$, then it is
fully faithful for $\star=\bs$.
 Furthermore, if the functor~$\rho^\bs$ is fully faithful,
then its essential image coincides with the full subcategory
$\sD^\bs_{I\ctra}(R\modl)\subset\sD^\bs(R\modl)$, because
the triangulated category $\sD^\bs(R\modl_{I\ctra}^\qs)$ is
generated by its abelian subcategory $R\modl_{I\ctra}^\qs$ and
the triangulated subcategory in $\sD^\bs(R\modl)$ generated by
$R\modl_{I\ctra}^\qs$ is precisely $\sD^\bs_{I\ctra}(R\modl)$
(in view of Lemma~\ref{I-contramodule-R-modules-lemma}
and Proposition~\ref{extension-of-two}).
 So the functor~$\xi^\bs$ is an equivalence of categories in this case.

 Now the functor of inclusion of abelian categories
$R\modl_{I\ctra}^\qs\rarrow R\modl$ has a left adjoint functor
$\boL_0\Lambda_I\:R\modl\rarrow R\modl_{I\ctra}^\qs$
(see Proposition~\ref{quotseparated-reflector}(c)).
 The functor $\boL_0\Lambda_I$ is left adjoint to an exact
functor, so it takes projective $R$\+modules to projective objects
of the category $R\modl_{I\ctra}^\qs$.
 One easily concludes that the functor $\rho^-\:
\sD^-(R\modl_{I\ctra}^\qs)\rarrow\sD^-(R\modl)$ has a left adjoint
functor $\lambda^-\:\sD^-(R\modl)\rarrow\sD^-(R\modl_{I\ctra}^\qs)$,
which is computable as the left derived functor of the functor
$\boL_0\Lambda_I$, or which is the same, the left derived functor
of the functor~$\Lambda_I$.
 So the functor~$\lambda^-$ assigns to a bounded above complex of
projective $R$\+modules $P^\bu$ the bounded above complex of projective
quotseparated $I$\+contramodule $R$\+modules $\Lambda_I(P^\bu)$.

 More generally, for $\star=\varnothing$, the functor
$$
 \rho\:\sD(R\modl_{I\ctra}^\qs)\lrarrow\sD(R\modl)
$$
also has a left adjoint functor
$$
 \lambda\:\sD(R\modl)\lrarrow\sD(R\modl_{I\ctra}^\qs),
$$
which can be computed by applying the functor $\Lambda_I$ to
homotopy projective complexes of $R$\+modules (known also as
``K\+projective complexes''~\cite{Spa}).
 Following~\cite[Proposition~3.6]{PSY}, one can also compute
the functor~$\lambda$ by applying the functor~$\Lambda_I$
to homotopy flat (``K\+flat'') complexes of $R$\+modules.

 For $\star=\bs$, the functor $\rho^\bs\:\sD^\bs(R\modl_{I\tors})\rarrow
\sD^\bs(R\modl)$ does not seem to necessarily have a left adjoint,
in general; so we are interested in its partially defined left
adjoint functor~$\lambda^\bs$.
 It is only important for us that the partial functor~$\lambda^\bs$
is defined on projective $R$\+modules $P\in R\modl_\proj\subset
R\modl\subset\sD^\bs(R\modl)$; indeed, one has $\lambda^\bs(P)=
\Lambda_I(P)\in R\modl_{I\ctra}^\qs\subset\sD^\bs(R\modl_{I\ctra}^\qs)$,
as can be easily seen.

 On the other hand, for any symbol $\star=\bs$, $+$, $-$, or
$\varnothing$, the fully faithful inclusion functor
$\iota^\star\:\sD^\star_{I\ctra}(R\modl)\rarrow\sD^\star(R\modl)$ has
a left adjoint functor $\eta^\star\:\sD^\star(R\modl)\rarrow
\sD^\star_{I\ctra}(R\modl)$.
 This holds quite generally for any finitely generated ideal
$I\subset R$; following the discussion in Section~\ref{seq-secn}
and~\cite[Section~3]{Pmgm}, the functor~$\eta^\star$ is computable as
$$
 \eta^\star(C^\bu)=\boR\Hom_R(K_\infty\spcheck(R;\s),C^\bu)
 \quad\text{for all $C^\bu\in\sD^\star(R\modl)$}.
$$

 Finally, if the functor $\xi^\bs:\sD^\bs(R\modl_{I\ctra}^\qs)\rarrow
\sD^\bs_{I\ctra}(R\modl)$ is a category equivalence, then it identifies
the functor~$\rho^\bs$ with the functor~$\iota^\bs$; and consequently
it also identifies their adjoint functors~$\lambda^\bs$ and~$\eta^\bs$.
 It follows that the functor~$\lambda^\bs$ is everywhere defined in
this case; but this is not important for us.
 The key observation is that the objects $\lambda^\bs(P)\in
\sD^\bs(R\modl_{I\ctra}^\qs)$ and $\eta^\bs(P)\in
\sD^\bs_{I\ctra}(R\modl)$ are identified by the functor~$\xi^\bs$,
for any projective $R$\+module~$P$.
 In other words, we have an isomorphism of complexes of $R$\+modules
$$
 \Lambda_I(P)\simeq\boR\Hom_R(K_\infty\spcheck(R;\s),P).
$$

 Hence $H^k\,\boR\Hom_R(K_\infty\spcheck(R;\s),P)=0$ for $k<0$ and
$\Lambda_I(P)\simeq\Delta_I(P)$ (see
isomorphism~\eqref{delta-computed}).
 It follows that the $R$\+module $\Delta_I(P)$ is $I$\+adically
separated, and therefore the canonical surjective morphism
$b_{I,P}\:\Delta_I(P)\rarrow\Lambda_I(P)$ is an isomorphism.
 According to Theorem~\ref{when-weakly-proregular}, we can conclude
that the ideal $I\subset R$ is weakly proregular.

 Alternatively, for the proof of the converse assertion one can
first observe that if the functor~$\rho^\bs$ is fully faithful, then
$R\modl_{I\ctra}^\qs=R\modl_{I\ctra}$ (in view of
Proposition~\ref{extension-of-two}).
 This translates, via Lemma~\ref{when-all-quotseparated}, into
the map $b_{I,R[[X]]}\:\Delta_I(R[X])\rarrow\Lambda_I(R[X])$
being an isomorphism.
 Secondly, one can apply the converse assertion of the next
Theorem~\ref{I-contramodule-theorem}; and finally use
the isomorphism~\eqref{delta-computed} and
Theorem~\ref{when-weakly-proregular}.
 This is the argument hinted at in Remark~\ref{main-remark}.
\end{proof}

 Our last theorem in this section applies to the category of
(not necessarily quotseparated) $I$\+contramodules.
 It is the result promised in Remark~\ref{main-remark}.
 The cohomology vanishing condition appearing in this theorem 
corresponds, in terms of
Proposition~\ref{weak-proregularity-piece-by-piece},
to the combination of conditions~(i) for all $q\le-2$ and
conditions~(ii) for all $q\le-1$.
 So it is a bit weaker than weak proregularity of the ideal~$I$,
in that condition~(i) for $q=-1$ is not required.
 We refer to Remark~\ref{main-remark} for the discussion.
 
 Once again, we need to introduce some notation.
 Denote the triangulated functor $\sD^\star(R\modl_{I\ctra})\rarrow
\sD^\star(R\modl)$ induced by the inclusion of abelian categories
$R\modl_{I\ctra}\rarrow R\modl$ by
$$
 \pi^\star\:\sD^\star(R\modl_{I\ctra})\lrarrow\sD^\star(R\modl).
$$
 The functor~$\pi^\star$ obviously factorizes into a composition
$$
 \sD^\star(R\modl_{I\ctra})\overset{\zeta^\star}\lrarrow
 \sD^\star_{I\ctra}(R\modl)\overset{\iota^\star}\lrarrow
 \sD^\star(R\modl).
$$

\begin{thm} \label{I-contramodule-theorem}
 If the cohomology vanishing condition
$H^k\,\boR\Hom_R(K_\infty\spcheck(R;\s),\>R[X])=0$ holds for some
infinite set $X$ and all $k<0$, then for every symbol\/
$\star=\bs$, $+$, $-$, or\/~$\varnothing$, the functor\/
$\zeta^\star\:\sD^\star(R\modl_{I\ctra})\rarrow
\sD^\star_{I\ctra}(R\modl)$ is an equivalence of categories and
the functor $\pi^\star\:\sD^\star(R\modl_{I\ctra})\rarrow
\sD^\star(R\modl)$ is fully faithful.

 Conversely, if for one of the symbols\/ $\star=\bs$, $+$, $-$,
or\/~$\varnothing$ the functor~$\pi^\star$ is fully faithful,
then $H^k\,\boR\Hom_R(K_\infty\spcheck(R;\s),\>R[X])=0$ for all
sets~$X$ and all $k<0$.
\end{thm}

\begin{proof}
 For the direct assertion, one has to follow the proofs
of~\cite[Theorem~2.9 and Corollary~2.10]{Pmgm} and convince oneself
that our present cohomology vanishing condition is sufficient for
the purposes of those proofs in lieu of the full weak proregularity.
 Indeed, one observes that the assertions of~\cite[Lemma~2.7]{Pmgm},
and consequently of~\cite[Lemma~2.8]{Pmgm}, remain valid under our
cohomology vanishing condition (while~\cite[Lemma~2.5]{Pmgm}
requires condition~(i) for $q=-1$, as~\cite[Example~2.6]{Pmgm} shows;
but this lemma is not needed for the proofs of~\cite[Theorem~2.9
and Corollary~2.10]{Pmgm}).
 In the notation of~\cite[proof of Lemma~2.7(a)]{Pmgm}, one can
consider projective or free $R$\+modules $F$, which is enough.

 The proof of the converse is similar to that of
Theorem~\ref{quotseparated-I-contramodule-theorem}.
  Clearly, if the functor~$\pi^\star$ is fully faithful for one of
the symbols $\star=\bs$, $+$, $-$, or~$\varnothing$, then it is
fully faithful for $\star=\bs$.
 It is also clear that if the functor~$\pi^\bs$ is fully faithful,
then its essential image coincides with the full subcategory
$\sD^\bs_{I\ctra}(R\modl)\subset\sD^\bs(R\modl)$; so
the functor~$\zeta^\bs$ is a category equivalence in this case.

 Now the functor of inclusion of abelian categories
$R\modl_{I\ctra}\rarrow R\modl$ has a left adjoint functor
$\Delta_I\:R\modl\rarrow R\modl_{I\ctra}$.
 The functor $\Delta_I$ is left adjoint to an exact functor, so
it takes projective $R$\+modules to projective objects
of the category $R\modl_{I\ctra}$.
 One easily concludes that the functor $\pi^-\:
\sD^-(R\modl_{I\ctra})\rarrow\sD^-(R\modl)$ has a left adjoint
functor $\delta^-\:\sD^-(R\modl)\rarrow\sD^-(R\modl_{I\ctra})$, which is
computable as the left derived functor of the functor~$\Delta_I$.
 So the functor~$\delta^-$ assigns to a bounded above complex of
projective $R$\+modules $P^\bu$ the bounded above complex of projective
$I$\+contramodule $R$\+modules $\Delta_I(P^\bu)$.
 Similarly, for $\star=\varnothing$, the functor
$\pi\:\sD(R\modl_{I\ctra}^\qs)\rarrow\sD(R\modl)$
also has a left adjoint functor
$\delta\:\sD(R\modl)\rarrow\sD(R\modl_{I\ctra}^\qs)$,
which can be computed by applying the functor $\Delta_I$ to
homotopy projective complexes of $R$\+modules.

 For $\star=\bs$, the functor $\pi^\bs\:\sD^\bs(R\modl_{I\tors})\rarrow
\sD^\bs(R\modl)$ does not seem to necessarily have a left adjoint,
in general; but we are interested in its partially defined left
adjoint functor~$\delta^\bs$.
 It is only important for us that the partial functor~$\delta^\bs$
is defined on projective $R$\+modules $P\in R\modl_\proj\subset
R\modl\subset\sD^\bs(R\modl)$; indeed, one has $\delta^\bs(P)=
\Delta_I(P)\in R\modl_{I\ctra}\subset\sD^\bs(R\modl_{I\ctra})$.

 On the other hand, the fully faithful functor $\iota^\star\:
\sD^\star_{I\ctra}(R\modl)\rarrow\sD^\star(R\modl)$ has a left
adjoint functor~$\eta^\star$, as it was explained in the proof of
of Theorem~\ref{quotseparated-I-contramodule-theorem}.
 Now if the functor $\zeta^\bs\:\sD^\bs(R\modl_{I\ctra})\rarrow
\sD^\bs_{I\ctra}(R\modl)$ is a category equivalence, then it identifies
the functor~$\pi^\bs$ with the functor~$\iota^\bs$; and consequently
it also identifies their adjoint functors~$\delta^\bs$ and~$\eta^\bs$.
 Thus the objects $\delta^\bs(P)\in\sD^\bs(R\modl_{I\ctra})$ and
$\eta^\bs(P)\in\sD^\bs_{I\ctra}(R\modl)$ are identified by
the functor~$\zeta^\bs$, for any projective $R$\+module~$P$.
 In other words, we have an isomorphism of complexes of $R$\+modules
$$
 \Delta_I(P)\simeq\boR\Hom_R(K_\infty\spcheck(R;\s),P).
$$
 Hence $H^k\,\boR\Hom_R(K_\infty\spcheck(R;\s),P)=0$ for $k<0$,
as desired.
\end{proof}

\Section{Idealistically Derived Complete Complexes}

 Let us recall the notation of
Theorem~\ref{quotseparated-I-contramodule-theorem} and its proof.
 The inclusion of abelian categories $R\modl_{I\ctra}^\qs
\rarrow R\modl$ induces a triangulated functor
$$
 \rho\:\sD(R\modl_{I\ctra}^\qs)\lrarrow\sD(R\modl),
$$
which has a left adjoint functor
$$
 \lambda\:\sD(R\modl)\lrarrow\sD(R\modl_{I\ctra}^\qs).
$$
 Porta--Shaul--Yekutieli in~\cite[Section~3]{PSY} and Yekutieli
in~\cite[Section~1]{Yek3} consider the functor
$$
  \boL\Lambda_I\:\sD(R\modl)\lrarrow\sD(R\modl),
$$
which is called the \emph{idealistic derived $I$\+adic completion
functor} in~\cite{Yek3}.
 Following the construction of the functor~$\lambda$ discussed
in the proof of Theorem~\ref{quotseparated-I-contramodule-theorem},
the functor $\boL\Lambda_I$ is the composition of the two adjoint
functors, $\boL\Lambda_I=\rho\circ\lambda$,
$$
 \sD(R\modl)\overset\lambda\lrarrow\sD(R\modl_{I\ctra}^\qs)
 \overset\rho\lrarrow\sD(R\modl).
$$

 The point is that applying the functor $\Lambda_I$ to every term
of a homotopy projective (or homotopy flat) complex of $R$\+modules
$P^\bu$ produces a complex of $I$\+adically separated and complete
$R$\+modules $\Lambda_I(P^\bu)$.
 All such modules are quotseparated $I$\+contramodules; so
a complex of such modules can be naturally considered as
an object of the derived category of quotseparated $I$\+contramodules,
$\lambda(P^\bu)=\Lambda_I(P^\bu)\in\sD(R\modl_{I\ctra}^\qs)$.
 One can also consider $\Lambda_I(P^\bu)$ as an object of the derived
category of $R$\+modules; this means applying the functor~$\rho$
to the object~$\lambda(P^\bu)$.

 Let us recall the definition from~\cite[Section~1]{Yek3}.
 A complex of $R$\+modules $C^\bu$ is said to be \emph{derived
$I$\+adically complete in the idealistic sense} if
the adjunction morphism
$$
 C^\bu\lrarrow\rho\lambda(C^\bu)=\boL\Lambda_I(C^\bu)
$$
is an isomorphism in $\sD(R\modl)$.
 Let us denote the full subcategory of derived $I$\+adically complete
complexes in the idealistic sense by
$$
 \sD(R\modl)^{ideal}_{I\com}\,\subset\,\sD(R\modl).
$$

 Similarly, one can denote by $\sD(R\modl_{I\ctra}^\qs)^{ideal}\subset
\sD(R\modl_{I\ctra}^\qs)$ the full subcategory of all 
complexes of quotseparated $I$\+contramodule $R$\+modules $B^\bu$
for which the adjunction morphism $\lambda\rho(B^\bu)\rarrow B^\bu$
is an isomorphism in $\sD(R\modl_{I\ctra}^\qs)$.
 Then $\sD(R\modl)^{ideal}_{I\com}\subset\sD(R\modl)$ and
$\sD(R\modl_{I\ctra}^\qs)^{ideal}\subset\sD(R\modl_{I\ctra}^\qs)$
are the maximal two full (triangulated) subcategories in
the respective triangulated categories in restriction to which
the functors~$\rho$ and~$\lambda$ are mutually inverse equivalences,
$$
 \sD(R\modl_{I\ctra}^\qs)\,\supset\,\sD(R\modl_{I\ctra}^\qs)^{ideal}
 \,\simeq\,\sD(R\modl)^{ideal}_{I\com}\,\subset\,\sD(R\modl).
$$

\begin{lem}
 For any finitely generated ideal $I\subset R$, all derived
$I$\+adically complete complexes in the idealistic sense are derived
$I$\+adically complete in the sequential sense.
\end{lem}

\begin{proof}
 All derived $I$\+adically complete complexes in the idealistic sense
belong to the essential image of the functor~$\rho$, which is
contained in the full subcategory $\sD_{I\ctra}(R\modl)\subset
\sD(R\modl)$ in view of
the factorization~\eqref{rho-xi-iota-factorization}.
 To put it simply, a derived $I$\+adically complete complex $C^\bu$
in the idealistic sense is quasi-isomorphic to a complex of
$I$\+adically separated and complete $R$\+modules
$\boL\Lambda_I(C^\bu)$, while the cohomology modules of any complex
of $I$\+adically separated and complete $R$\+modules are
(quotseparated) $I$\+contramodule $R$\+modules.
 It remains to recall that, by Lemma~\ref{sequentially-complete},
\,$\sD_{I\ctra}(R\modl)\subset\sD(R\modl)$ is precisely the full
subcategory of derived $I$\+adically complete complexes in
the sequential sense.
\end{proof}

\begin{prop} \label{RHom-not-idealistic-complete}
 Let $X$ be an infinite set and $B^\bu$ be the complex (derived
category object) $B^\bu=\boR\Hom_R(K_\infty\spcheck(R;\s),\>R[X])
\in\sD(R\modl)$.
 Then \par
\textup{(a)} the complex $B^\bu$ is derived $I$\+adically complete
in the sequential sense; \par
\textup{(b)} the ideal $I\subset R$ is weakly proregular if and only if
the complex $B^\bu$ is derived $I$\+adically complete in the idealistic
sense.
\end{prop}

\begin{proof}
 Part~(a): for any complex of $R$\+modules $C^\bu$, the complex/derived
category object $\boR\Hom_R(K_\infty\spcheck(R;\s),C^\bu)$ is
derived $I$\+adically complete in the sequential sense.
 In fact, $\eta=\boR\Hom_R(K_\infty\spcheck(R;\s),{-})$ is
the reflector $\eta\:\sD(R\modl)\rarrow\sD_{I\ctra}(R\modl)$; see
the discussion in the proof of
Theorem~\ref{quotseparated-I-contramodule-theorem}
and in~\cite[Section~3]{Pmgm}.
 In other words, it suffices to recall that the sequential derived
$I$\+adic completion functor is idempotent; see~\cite[Remark~2.23]{Yek3}
and the end of Section~\ref{seq-secn}.

 Part~(b): if the ideal $I$ is weakly proregular, then a complex of
$R$\+modules is derived $I$\+adically complete in the idealistic sense
if and only if it is derived $I$\+adically complete in the sequential
sense (see~\cite[Corollary~3.12]{Yek3}, which is based
on~\cite[Theorem~3.11\,(i)\,$\Rightarrow$\,(iii)]{Yek3} and
the results of~\cite{PSY}; or alternatively
Theorem~\ref{quotseparated-I-contramodule-theorem} above).
 Keeping part~(a) in mind, this proves the ``only if'' assertion.

 ``If'': in the notation of
Theorem~\ref{quotseparated-I-contramodule-theorem} and its proof,
we have $B^\bu=\eta(R[X])=\iota\eta(R[X])$.
 Notice that the functor left adjoint to the functor $\xi\:
\sD(R\modl_{I\ctra}^\qs)\rarrow\sD_{I\ctra}(R\modl)$ is computable as
the composition $\sD_{I\ctra}(R\modl)\overset\iota\rarrow
\sD(R\modl)\overset\lambda\rarrow\sD(R\modl_{I\ctra}^\qs)$, since
the functor~$\iota$ is fully faithful.
 It follows that the composition of functors $(\lambda\iota)\eta$ is
left adjoint to the functor $\iota\xi=\rho$; hence
$\lambda\iota\eta=\lambda$.

 Now assume that the complex $B^\bu$ is derived $I$\+adically complete
in the idealistic sense.
 Then the adjunction morphism
$$
 \iota\eta(R[X])=B^\bu\lrarrow\rho\lambda(B^\bu)=
 \rho\lambda\iota\eta(R[X])=\rho\lambda(R[X])
$$
is an isomorphism in $\sD(R\modl)$.
 In other words, this means that the canonical morphism
$$
 \boR\Hom_R(K_\infty\spcheck(R;\s),\>R[X])=\iota\eta(R[X])\lrarrow
 \rho\lambda(R[X])=\Lambda_I(R[X])
$$
is an isomorphism in $\sD(R\modl)$.
 According to Theorem~\ref{when-weakly-proregular}, it follows that
the ideal $I\subset R$ is weakly proregular.
\end{proof}

\begin{prop} \label{decaying-functions-not-idealistic-complete}
 Let $X$ be an infinite set and $Q_X=\R[[X]]=\Lambda_I(R[X])$ be
the $R$\+module of decaying functions $X\rarrow\R=\Lambda_I(R)$
(in the sense of\/~\cite{Yek0}).
 Then \par
\textup{(a)} the $R$\+module $Q_X$ is a derived $I$\+adically complete
complex in the sequential sense; \par
\textup{(b)} the ideal $I\subset R$ is weakly proregular if and only
if the $R$\+module $Q_X$ is a derived $I$\+adically complete complex
in the idealistic sense.
\end{prop}

\begin{proof}
 Part~(a): the $R$\+module $Q_X$ is a quotseparated $I$\+contramodule;
in fact, $Q_X$ is a projective (in some sense, free) object in
the abelian category $R\modl_{I\ctra}^\qs$ (see the discussion in
Section~\ref{derived-complete-modules-secn}).
 Hence $Q_X\in R\modl_{I\ctra}^\qs\subset R\modl_{I\ctra}\subset
\sD_{I\ctra}(R\modl)$ is a derived $I$\+adically complete complex in
the sequential sense by Lemma~\ref{sequentially-complete}.

 Part~(b): the ``only if'' assertion follows from part~(a) for
the same reason as in the proof of
Proposition~\ref{RHom-not-idealistic-complete}.
 To prove the ``if'', let us consider the $R$\+module $Q_X$ as
an object of the derived category $\sD(R\modl_{I\ctra}^\qs)$.
 Then the same $R$\+module viewed as an object of the derived
category $\sD(R\modl)$ is denoted by~$\rho(Q_X)$.
 Assume that $\rho(Q_X)$ is a derived $I$\+adically complete complex in
the idealistic sense; then the adjunction morphism $\rho(Q_X)\rarrow
\rho\lambda\rho(Q_X)$ is an isomorphism.

 For any pair of adjoint functors $\rho$ and~$\lambda$ and any
object $Q$ in the relevant category, the composition of natural
morphisms $\rho(Q)\rarrow \rho\lambda\rho(Q)\rarrow\rho(Q)$ is
the identity morphism.
 If the morphism $\rho(Q)\rarrow\rho\lambda\rho(Q)$ is an isomorphism,
then the morphism $\rho\lambda\rho(Q)\rarrow\rho(Q)$ is
an isomorphism, too.

 The latter morphism is obtained by applying the functor~$\rho$ to
the adjunction morphism $\lambda\rho(Q)\rarrow Q$.
 In the situation at hand, the functor~$\rho$ is conservative: if
$f\:B^\bu\rarrow C^\bu$ is a morphism in the derived category
$\sD(R\modl_{I\ctra}^\qs)$ and $\rho(f)\:\rho(B^\bu)\rarrow\rho(C^\bu)$
is an isomorphism in $\sD(R\modl)$, then the morphism~$f$ is
an isomorphism (since the inclusion of abelian categories
$R\modl_{I\ctra}^\qs\rarrow R\modl$ is an exact functor taking nonzero
objects to nonzero objects).
 We conclude that the adjunction morphism $\lambda\rho(Q_X)\rarrow Q_X$
is an isomorphism in $\sD(R\modl_{I\ctra}^\qs)$.

 At this point we have to recall that the object $Q_X$, by definition,
depends on the choice of an infinite set~$X$.
 We claim that if the morphism $\lambda\rho(Q_X)\rarrow Q_X$ is
an isomorphism for one particular infinite set $X$, then so is
the adjunction morphism $\lambda\rho(Q_Y)\rarrow Q_Y$ for every set~$Y$.
 Indeed, passing to a direct summand, one can assume the set $X$ to be
countable.
 The key observation is that both the inclusion functor
$R\modl_{I\ctra}^\qs\rarrow R\modl$ and the $I$\+adic completion functor
$\Lambda_I\:R\modl\rarrow R\modl_{I\ctra}^\qs$ preserve
countably-filtered direct limits.
 Using functorial (homotopy) flat resolutions together with
the observation that the derived functor~$\lambda$ can be computed
with (homotopy) flat resolutions~\cite[Lemma~3.5 and
Proposition~3.6]{PSY}, one shows that both the functors~$\rho$
and~$\lambda$ preserve countably-filtered direct limits of complexes
of modules (in an appropriate sense).
 It remains to observe that $Q_Y=\varinjlim_{Z\subset Y}Q_Z$, where
the direct limit is taken over all the countable subsets $Z$ of
a given infinite set~$Y$.
 We suggest~\cite{AR} as a background reference on
$<\aleph_1$\+filtered colimits; and leave it to the reader to fill
the details of the above argument.

 In the remaining last paragraph of this proof, the argument is based
on the premise, justified above, that the morphism $\lambda\rho(Q_Y)
\rarrow Q_Y$ is an isomorphism for every set~$Y$.
 Denote the abelian categories involved by $\sA=R\modl$ and
$\sB=R\modl_{I\ctra}^\qs$; and let $B\in\sB$ be an arbitrary object.
 Viewing $B$ as an object of the derived category $\sD(\sB)$, we
can compute
$$
 \Hom_{\sD(\sA)}(\rho(Q_Y),\rho(B)[k])\simeq
 \Hom_{\sD(\sB)}(\lambda\rho(Q_Y),B[k])\simeq
 \Hom_{\sD(\sB)}(Q_Y,B[k])=0
$$
for all $B\in\sB$ and $k>0$.
 The next lemma allows to conclude that the triangulated functor
$\rho^\bs\:\sD^\bs(R\modl_{I\ctra}^\qs)\rarrow\sD^\bs(R\modl)$ is
fully faithful.
 By Theorem~\ref{quotseparated-I-contramodule-theorem}, it follows
that the ideal $I\subset R$ is weakly proregular.
\end{proof}

\begin{lem} \label{derived-full-and-faithfulness-lemma}
 Let $\rho^\as\:\sB\rarrow\sA$ be a fully faithful exact functor
between two abelian categories, and let $\rho^\bs\:\sD^\bs(\sB)\rarrow
\sD^\bs(\sA)$ be the induced triangulated functor between the bounded
derived categories.
 Then the following two conditions are equivalent:
\begin{enumerate}
\item the functor~$\rho^\bs$ is fully faithful;
\item for any two objects $C$ and $B\in\sB$ and all integers $k>0$,
the functor~$\rho^\as$ induces isomorphisms of the\/ $\Ext$ groups
$$
 \Ext_\sB^k(C,B)\simeq\Ext_\sA^k(\rho^\as(C),\rho^\as(B)).
$$
\end{enumerate}
 Furthermore, if there are enough projective objects in
the abelian category\/ $\sB$, then conditions~\textup{(1\+-2)} are
equivalent to
\begin{enumerate}
\setcounter{enumi}{2}
\item for any projective object $Q\in\sB$, any object $B\in\sB$, and
all integers $k>0$, one has
$$
 \Ext_\sA^k(\rho^\as(Q),\rho^\as(B))=0.
$$
\end{enumerate}
\end{lem}

\begin{proof}
 The equivalence (1)\,$\Longleftrightarrow$\,(2) holds because
the triangulated category $\sD^\bs(\sB)$ is generated by its full
subcategory $\sB\subset\sD(\sB)$.
 To check the implication (2)\,$\Longrightarrow$\,(3), it suffices
to take $C=Q$.
 In order to prove (3)\,$\Longrightarrow$\,(2), one can
choose a projective resolution $Q_\bu$ of the object $C$ in
the category $\sB$ and compute that $\Ext^k_\sB(C,B)=
H^k\Hom_\sB(Q_\bu,B)\simeq H^k\Hom_\sA(\rho^\as(Q_\bu),\rho^\as(B))
\simeq\Ext^k_\sA(\rho^\as(C),\rho^\as(B))$, where the latter
isomorphism holds since $\rho^\as(Q_\bu)$ is a resolution of
the object $\rho^\as(C)\in\sA$ by objects $\rho^\as(Q_i)$, \,$i\ge0$,
satisfying $\Ext^k_\sA(\rho^\as(Q_i),\rho^\as(B))=0$ for $k>0$.
\end{proof}

\begin{cor} \label{idealistic-completion-idempotence}
 The idealistic derived $I$\+adic completion functor\/
$\boL\Lambda_I\:\sD(R\modl)\rarrow\sD(R\modl)$ is idempotent if and
only if the ideal $I\subset R$ is weakly proregular.  \hfuzz=1.2pt
\end{cor}

\begin{proof}
 ``If'': for a weakly proregular ideal $I$, the functor $\boL\Lambda_I$
is isomorphic to the sequential derived $I$\+adic completion functor
$\boR\Hom_R(K_\infty\spcheck(R;\s),{-})\:\sD(R\modl)\rarrow
\sD(R\modl)$ \cite[Corollary~5.25]{PSY},
\cite[Theorem~3.11\,(i)\,$\Rightarrow$\,(iii)]{Yek3}.
 The sequential derived $I$\+adic completion functor is always
idempotent (see~\cite[Remark~2.23]{Yek3} or the discussion at
the end of Section~\ref{seq-secn}).

 ``Only if'': for any $R$\+module $C$, the adjunction morphism $C\rarrow
\boL\Lambda_I(C)$ induces the adjunction morphism $C\rarrow
\boL_0\Lambda_I(C)$ after the passage to the degree-zero cohomology
modules (cf.\ Proposition~\ref{quotseparated-reflector}(c)).
 In particular, the $R$\+module $C$ is a quotseparated $I$\+contramodule
if and only if the induced morphism $C\rarrow H^0(\boL\Lambda_I(C))=
\boL_0\Lambda_I(C)$ is an isomorphism.
 Hence the object $C\in\sD(R\modl)$ is derived $I$\+adically complete
in the idealistic sense if and only if $C$ is a quotseparated
$I$\+contramodule and the $R$\+modules $\boL_i\Lambda_I(C)=
H^{-i}(\boL\Lambda_I(C))$ vanish for all $i>0$.
 It follows that existence of an isomorphism $C\simeq\boL\Lambda_I(C)$
in the derived category $\sD(R\modl)$ implies that $C$ is derived
$I$\+adically complete in the idealistic sense.

 Now the functor $\boL\Lambda_I$ takes the free $R$\+module
with a countable set of generators $P_X=R[X]\in\sD(R\modl)$ to
the $R$\+module of decaying functions $Q_X=\R[[X]]\in\sD(R\modl)$.
 According to the previous paragraph, if the object $\boL\Lambda_I(Q_X)
\in\sD(R\modl)$ is isomorphic to $Q_X$ (i.~e., some isomorphism
exists), then the object $Q_X\in\sD(R\modl)$ is derived $I$\+adically
complete in the idealistic sense.
 According to
Proposition~\ref{decaying-functions-not-idealistic-complete}(b), it
follows that the ideal $I\subset R$ is weakly proregular.
\end{proof}

 Proposition~\ref{decaying-functions-not-idealistic-complete} appears to
confirm the feeling that the class of derived $I$\+adically complete
complexes in the idealistic sense, as defined in the paper~\cite{Yek3},
may be too small (when the ideal $I\subset R$ is not weakly proregular).
 Is there \emph{any} example of a nonzero derived $I$\+adically complete
complex in the idealistic sense, for an arbitrary finitely generated
ideal $I\ne R$ in a commutative ring~$R$\,?

 We would suggest the derived category $\sD(R\modl_{I\ctra}^\qs)$
as a proper replacement of the category of derived $I$\+adically
complete complexes in the idealistic sense.
 At least, the category $\sD(R\modl_{I\ctra}^\qs)$ contains
the module of decaying functions $\R[[X]]=\Lambda_I(R[X])$
as an object.

\begin{rem}
 The complicated behavior of the idealistic derived completion
functor $\boL\Lambda_I$ can be conceptualized in the following way.
 Suppose that we have chosen a finite set of generators $s_1$,~\dots,
$s_m$ for the ideal~$I$.
 Then the $I$\+adic completion $\Lambda_I(M)$ of an $R$\+module $M$,
viewed as a $\boZ[s_1,\dotsc,s_m]$\+module, is determined by
the underlying $\boZ[s_1,\dotsc,s_m]$\+module structure on $M$, and does
not depend on the $R$\+module structure.
 Still the ring $R$ dictates what kind of resolutions (viz., homotopy
flat or homotopy projective complexes of $R$\+modules, or if one wishes,
homotopy projective complexes of projective $R$\+modules) are to be used
in the construction of the derived functor~$\boL\Lambda_I$.
 When the $\boZ[s_1,\dotsc,s_m]$\+module structure of the ring $R$ is
complicated, so is the resulting derived functor.

 One can avoid this problem by resolving the ring $R$ itself.
 This approach was developed in the paper~\cite{Sh}.
 Following~\cite{Sh}, let $B^\bu$ be a graded commutative, nonnegatively
cohomologically graded DG\+ring; so in particular, $B^0$ is
a commutative ring, $B^\bu$ is a complex of $B^0$\+modules, and
$H^0(B^\bu)$ is a quotient ring of~$B^0$.
 Assume that $B^\bu$ is a homotopy flat complex of $B^0$\+modules,
and let $c_1$,~\dots, $c_m\in B^0$ be a finite sequence of elements
generating a weakly proregular ideal $E\subset B^0$.
 Let $\sD(B^\bu\modl)$ denote the derived category of (unbounded)
DG\+modules over~$B^\bu$.
 The derived $E$\+adic completion functor $\boL\Lambda_E\:
\sD(B^\bu\modl)\rarrow\sD(B^\bu\modl)$ is constructed by applying
the $E$\+adic completion functor $\Lambda_E\:M^\bu\longmapsto
\varprojlim_{n\ge1}M^\bu/E^nM^\bu$ to homotopy projective
or homotopy flat DG\+modules over~$B^\bu$.

 According to~\cite[Proposition~2.4]{Sh}, under the assumptions above
the (idealistic) derived $E$\+adic completion functor $\boL\Lambda_E$
is isomorphic to the sequential derived $E$\+adic completion functor
$\boR\Hom_{B^0}(K_\infty\spcheck(B^0;\mathbf c),{-})\:
\sD(B^\bu\modl)\rarrow\sD(B^\bu\modl)$ (where $\mathbf c$~is
a notation for the sequence $c_1$,~\dots,~$c_m$).
 It follows that $\boL\Lambda_E$ is an idempotent functor
$\sD(B^\bu\modl)\rarrow\sD(B^\bu\modl)$ \,\cite[Proposition~2.10]{Sh}.   
 The essential image of $\boL\Lambda_E$ is the full subcategory
$\sD_{\overline E\ctra}(B^\bu\modl)\subset\sD(B^\bu\modl)$ of all
DG\+modules whose cohomology $H^0(B^\bu)$\+modules are 
$\overline E$\+contra\-modules, where $\overline E\subset H^0(B^\bu)$
is the image of the ideal $E\subset B^0$ under the surjective ring
homomorphism $B^0\rarrow H^0(B^\bu)$. 

 Now given a ring $R$ with a finitely generated ideal $I$, one can
construct a DG\+ring $B^\bu$ with a finitely generated ideal
$E\subset B^0$ satisfying the assumptions above, together with
a quasi-isomorphism of DG\+rings $B^\bu\rarrow R$ such that
$I=\overline E$ (this is a particular case 
of~\cite[Proposition~2.2]{Sh}).
 Then the derived category $\sD(R\modl)$ is equivalent to
$\sD(B^\bu\modl)$, so one can view $\boL\Lambda_E$ as a functor
$\sD(R\modl)\rarrow\sD(R\modl)$.
 This functor is isomorphic to the sequential derived completion
functor $\boR\Hom_R(K_\infty\spcheck(R,\s),{-})\:\sD(R\modl)\rarrow
\sD(R\modl)$ (where $s_1$,~\dots, $s_m\in R$ are the images of
the elements $c_1$,~\dots, $c_m\in B^0$ under the surjective
ring homomorphism $B^0\rarrow H^0(B^\bu)=R$).
 The triangulated equivalence $\sD(B^\bu\modl)\simeq\sD(R\modl)$
identifies the full subcategory $\sD_{\overline E\ctra}(B^\bu\modl)
\subset\sD(B^\bu\modl)$ with the full subcategory
$\sD_{I\ctra}(R\modl)\subset\sD(R\modl)$.
\end{rem}

\Section{Idealistically Derived Torsion Complexes}

 The results of this section are dual-analogous to those of
the previous one.
 To begin with, let us recall the notation of
Theorem~\ref{I-torsion-theorem}.
 The inclusion of abelian categories $R\modl_{I\tors}\rarrow R\modl$
induces a triangulated functor
$$
 \mu\:\sD(R\modl_{I\tors}) \lrarrow \sD(R\modl),
$$
which has a right adjoint functor
$$
 \gamma\:\sD(R\modl)\lrarrow\sD(R\modl_{I\tors}).
$$
 Porta--Shaul--Yekutieli in~\cite[Section~3]{PSY} and
Yekutieli in~\cite[Section~1]{Yek3} consider the functor
$$
 \boR\Gamma_I\:\sD(R\modl)\lrarrow\sD(R\modl),
$$
which is called the \emph{idealistic derived $I$\+torsion functor}
in~\cite{Yek3}.
 Following the construction of the functor~$\gamma$ discussed in
the proof of Theorem~\ref{I-torsion-theorem}, the functor $\boR\Gamma_I$
is the composition of the two adjoint functors, $\boR\Gamma_I=
\mu\circ\gamma$,
$$
 \sD(R\modl)\overset\gamma\lrarrow\sD(R\modl_{I\tors})
 \overset\mu\lrarrow\sD(R\modl).
$$

 The point is that applying the functor $\Gamma_I$ to every term of
a homotopy injective complex of $R$\+modules $J^\bu$ produces a complex
of $I$\+torsion $R$\+modules.
 Such a complex can be naturally considered as an object of
the category $\sD(R\modl_{I\tors})$.
 One can also view $\Gamma_I(J^\bu)$ as an object of the derived
category of $R$\+modules; this means applying the functor~$\mu$
to the object~$\gamma(J^\bu)$.

\begin{rem}
 The homotopy projective (``K\+projective''), homotopy injective, and
suchlike resolutions are quintessentially the technique for
constructing unbounded derived functors of infinite homological
dimension.

 The sequential derived torsion and completion functors clearly have
finite homological dimension, not exceeding the minimal number of
generators of the ideal $I\subset R$.
 Hence, in the weakly proregular case, the homological dimension of
the idealistic derived torsion and completion functors is also
finite~\cite[Corollaries~4.28 and~5.27]{PSY}.
 We \emph{do not know} whether the homological dimensions of
the idealistic derived torsion and completion functors need to be
finite for an arbitrary finitely generated ideal $I\subset R$, but
we do not expect them to be.
 It is precisely for this reason that we (following~\cite{PSY}
and~\cite{Yek3}) are using homotopy projective/flat/injective
resolutions in the constructions of the idealistic derived functors.

 If we knew these functors to have finite homological dimension,
homotopy adjusted resolutions would not be needed for their
construction, as the discussion in~\cite[Sections~1 and~2]{Pmgm}
illustrates (see~\cite[Lemma~1.5]{VY} for a precise statement).
\end{rem}

 Following~\cite[Section~1]{Yek3}, a complex of $R$\+modules $M^\bu$ is
said to be \emph{derived $I$\+torsion in the idealistic sense} if
the adjunction morphism
$$
 \boR\Gamma_I(M^\bu)=\mu\gamma(M^\bu)\lrarrow M^\bu
$$
is an isomorphism in $\sD(R\modl)$.
 Let us denote the full subcategory of derived $I$\+torsion complexes
in the idealistic sense by
$$
 \sD(R\modl)^{ideal}_{I\tors}\,\subset\,\sD(R\modl)
$$

 Similarly, one can denote by $\sD(R\modl_{I\tors})^{ideal}$ the full
subcategory of all complexes of $I$\+torsion $R$\+modules $N^\bu$
for which the adjunction morphism $N^\bu\rarrow\gamma\mu(N^\bu)$ is
an isomorphism in $\sD(R\modl_{I\tors})$.
 Then $\sD(R\modl)_{I\tors}^{ideal}\subset\sD(R\modl)$ and
$\sD(R\modl_{I\tors})^{ideal}\subset\sD(R\modl_{I\tors})$ are
the maximal two full (triangulated) subcategories in the respective
triangulated categories in restriction to which the functors~$\mu$
and~$\gamma$ are mutually inverse equivalences,
$$
 \sD(R\modl_{I\tors})\,\supset\,\sD(R\modl_{I\tors})^{ideal}
 \,\simeq\,\sD(R\modl)_{I\tors}^{ideal}\,\subset\,\sD(R\modl).
$$
 
\begin{lem}
 For any finitely generated ideal $I\subset R$, all derived $I$\+torsion
complexes in the idealistic sense are derived $I$\+torsion in
the sequential sense.
\end{lem}

\begin{proof}
 All derived $I$\+torsion complexes in the idealistic sense belong to
the essential image of the functor~$\mu$, which is contained in
the full subcategory $\sD_{I\tors}(R\modl)\subset\sD(R\modl)$ in view
of the factorization~\eqref{mu-chi-upsilon-factorization}.
 Simply put, a derived $I$\+torsion complex $M^\bu$ in the idealistic
sense is quasi-isomorphic to a complex of $I$\+torsion $R$\+modules
$\boR\Gamma_I(M^\bu)$, whose cohomology modules are obviously
$I$\+torsion, too.
 It remains to recall that, by Lemma~\ref{sequentially-torsion},
\,$\sD_{I\tors}(R\modl)\subset\sD(R\modl)$ is precisely the full
subcategory of derived $I$\+torsion complexes in the sequential sense.
\end{proof}

\begin{prop}
 Given an injective $R$\+module $J$, consider the complex of
$R$\+modules $N^\bu=K_\infty\spcheck(R;\s)\ot_R J$.
 Then \par
\textup{(a)} the complex $N^\bu$ is derived $I$\+torsion in
the sequential sense; \par
\textup{(b)} the ideal $I\subset R$ is weakly proregular if and only if
the complex $N^\bu$ is derived $I$\+torsion in the idealistic sense for
every injective $R$\+module~$J$.
\end{prop}

\begin{proof}
 Part~(a): for any complex of $R$\+modules $M^\bu$, the complex
$K_\infty\spcheck(R;\s)\ot_R M^\bu$ is derived $I$\+torsion in
the sequential sense, because the sequential derived $I$\+torsion
functor $K_\infty\spcheck(R;\s)\ot_R{-}$ is idempotent
(see formula~\eqref{Koszul-tensor-idempotent} in the end
of Section~\ref{seq-secn}).

 If the ideal $I$ is weakly proregular, then a complex of $R$\+modules
is derived $I$\+torsion in the idealistic sense if and only if it is
derived $I$\+torsion in the sequential sense
(see~\cite[Corollary~4.26]{Pmgm} and~\cite[Section~3]{Yek3},
or Theorem~\ref{I-torsion-theorem} above).
 Hence the ``only if'' implication in part~(b) follows from part~(a).

 Part~(b), ``if'': in the notation of Theorem~\ref{I-torsion-theorem}
and its proof, we have $N^\bu=\theta(J)=\upsilon\theta(J)$.
 Notice that the right adjoint functor to the functor
$\chi\:\sD(R\modl_{I\tors})\rarrow\sD_{I\tors}(R\modl)$ is computable
as the composition $\sD_{I\tors}(R\modl)\overset\upsilon\rarrow
\sD(R\modl)\overset\gamma\rarrow\sD(R\modl_{I\tors})$, since
the functor~$\upsilon$ is fully faithful.
 It follows that the composition of functors $(\gamma\upsilon)\theta$
is right adjoint to the functor $\upsilon\chi=\mu$; hence
$\gamma\upsilon\theta=\gamma$.

 Now assume that the complex $N^\bu$ is derived $I$\+torsion in
the idealistic sense for every injective $R$\+module~$J$.
 Then the adjunction morphism
$$
 \mu\gamma(J)=\mu\gamma\upsilon\theta(J)=\mu\gamma(N^\bu)\lrarrow
 N^\bu=\upsilon\theta(J)
$$
is an isomorphism in $\sD(R\modl)$.
 In other words, this means that the canonical morphism
$$
 \Gamma_I(J)=\mu\gamma(J)\lrarrow\upsilon\theta(J)=
 K_\infty\spcheck(R;\s)\ot_R J
$$
is an isomorphism in $\sD(R\modl)$.
 According to
Corollary~\ref{when-weakly-proregular-in-terms-of-injectives},
it follows that the ideal $I\subset R$ is weakly proregular.
\end{proof}

 The following result, which is dual-analogous to
Proposition~\ref{decaying-functions-not-idealistic-complete},
was obtained by Vyas and Yekutieli in~\cite[Theorem~0.3]{VY}.
 We provide a sketch of proof using our methods.

\begin{prop} \label{torsion-injective-not-idealistic-torsion}
 Given an injective $R$\+module $J$, consider the $R$\+module
$E=\Gamma_I(J)$.
 Then \par
\textup{(a)} the $R$\+module $E$ is a derived $I$\+torsion complex
in the sequential sense; \par
\textup{(b)} the ideal $I\subset R$ is weakly proregular if and only
if the $R$\+module $E$ is a derived $I$\+torsion complex in
the idealistic sense for every injective $R$\+module~$J$.
\end{prop}

\begin{proof}[Sketch of proof]
 Part~(a): the $R$\+module $E$ is $I$\+torsion by definition;
in fact, $E$ is an injective object in the abelian category
$R\modl_{I\tors}$.
 Hence $E\in R\modl_{I\tors}\subset\sD_{I\tors}(R\modl)$ is
a derived $I$\+torsion complex in the sequential sense
by Lemma~\ref{sequentially-torsion}.
 The ``only if'' assertion in part~(b) follows from part~(a).

 Part~(b), ``if'': let us consider the $R$\+module $E$ as an object
of the derived category $\sD(R\modl_{I\tors})$.
 Then the same $R$\+module viewed as an object of the derived
category $\sD(R\modl)$ is denoted by $\mu(E)$.
 Assume that $\mu(E)$ is a derived $I$\+torsion complex in
the idealistic sense, for every injective $R$\+module~$J$.
 Then the adjunction morphism $\mu\gamma\mu(E)
\rarrow\mu(E)$ is an isomorphism.

 It follows that the morphism $\mu(E)\rarrow\mu\gamma\mu(E)$ obtained
by applying~$\mu$ to the adjunction morphism $E\rarrow\gamma\mu(E)$
is an isomorphism, too.
 Since the triangulated functor $\mu\:\sD(I\modl_{I\tors})\rarrow
\sD(R\modl)$ is conservative (taking complexes with nonzero cohomology
to complexes with nonzero cohomology), we can conclude that
the adjunction morphism $E\rarrow\gamma\mu(E)$ is an isomorphism in
$\sD(R\modl_{I\tors})$.

 Denote the abelian categories involved by $\sA=R\modl$ and
$\sT=R\modl_{I\tors}$; and let $T\in\sT$ be an arbitrary object.
 Viewing $T$ as an object of the derived category $\sD(\sT)$,
we can compute
$$
 \Hom_{\sD(\sA)}(\mu(T),\mu(E)[k])\simeq
 \Hom_{\sD(\sT)}(T,\gamma\mu(E)[k])\simeq
 \Hom_{\sD(\sT)}(T,E[k])=0
$$
for all $T\in\sT$ and $k>0$.

 There are enough injectives in the abelian category
$\sT=R\modl_{I\tors}$, and any injective object $L$ in $\sT$ is a direct
summand of an object $E=\Gamma_I(J)$ for some injective $R$\+module~$J$.
 Hence we have $\Ext^k_\sA(T,L)=0$ for all objects $T\in\sT$,
all injective objects $L\in\sT$, and all $k>0$.
 By the dual version of Lemma~\ref{derived-full-and-faithfulness-lemma},
we can conclude that the triangulated functor $\mu^\bs\:
\sD^\bs(R\modl_{I\tors})\rarrow\sD^\bs(R\modl)$ is fully faithful.
 According to Theorem~\ref{I-torsion-theorem}, it follows that
the ideal $I\subset R$ is weakly proregular.
\end{proof}

\begin{cor}
 The idealistic derived $I$\+torsion functor\/ $\boR\Gamma_I\:
\sD(R\modl)\rarrow\sD(R\modl)$ is idempotent if and only if
the ideal $I\subset R$ is weakly proregular.
\end{cor}

\begin{proof}
 Dual-analogous to Corollary~\ref{idealistic-completion-idempotence}.
\end{proof}

 Once again, the feeling is that the class of derived $I$\+torsion
complexes in the idealistic sense, as defined in the paper~\cite{Yek3},
may be too small.
 We ask the same question: is there \emph{any} example of a nonzero
derived $I$\+torsion complex in the idealistic sense, for an arbitrary
finitely generated ideal $I\ne R$ in a commutative ring~$R$\,?

 We would suggest the derived category $\sD(R\modl_{I\tors})$ as
a proper replacement of the category of derived $I$\+torsion complexes
in the idealistic sense.
 At least, the category $\sD(R\modl_{I\tors})$ contains all
$I$\+torsion $R$\+modules as objects.

\Section{Digression: Adic Flatness and Weak Proregularity}

 In this section, we use the occasion to provide the precise
formulation and a proof of a result of the present author mentioned
by Yekutieli in~\cite[Remark~4.12]{Yek2}.
 This is closely related to the main results of this paper on
the technical level.

 Following~\cite[Definition~4.2]{Yek2}, we say that an $R$\+module
$F$ is \emph{$I$\+adically flat} if $\Tor^R_k(N,F)=0$ for all
$I$\+torsion $R$\+modules $N$ and all $k>0$.
 (In the terminology of~\cite[Definition~2.6.1]{SS}, such modules
are called ``relatively-$I$-flat''.)

 In the same spirit, the $R$\+modules satisfying the equivalent
conditions of the next proposition could be called ``$I$\+adically
projective''.

\begin{prop} \label{adically-projective}
 Let $F$ be an $R$\+module.
 Then the following three conditions are equivalent:
\begin{enumerate}
\item the $R$\+module $F$ is $I$\+adically flat and the $R/I$\+module
$F/IF$ is projective;
\item $\Ext^k_R(F,D)=0$ for all $R/I$\+modules $D$ and all $k>0$;
\item $\Ext^k_R(F,C)=0$ for all $I$\+contramodule $R$\+modules $C$ and
all $k>0$.
\end{enumerate}
\end{prop}

\begin{proof}
 (1)~$\Longleftrightarrow$~(2)
 Let us first prove that (2)~implies the $I$\+adic flatness of~$F$.
 First of all, since the functor $\Tor$ preserves direct limits, it
suffices to check that $\Tor^R_k(N,F)=0$ for all $k>0$ and all
$R/I^n$\+modules $N$, where $n$~ranges over the positive integers.
 Next one easily reduces to the case $n=1$; so we can assume that
$N$ is an $R/I$\+module.
 It remains to consider the character module $D=N^+=
\Hom_\boZ(N,\boQ/\boZ)$ and use the natural isomorphism
$\Tor^R_k(N,F)^+\simeq\Ext_R^k(F,N^+)$.

 Now we can assume that the $R$\+module $F$ is $I$\+adically flat;
in particular, $\Tor^R_i(R/I,F)=0$ for all $i>0$.
 Then, for any $R/I$\+module $D$ and all $k\ge0$, there is a natural
isomorphism of $\Ext$ modules
$$
 \Ext^k_R(F,D)\simeq\Ext^k_{R/I}(F/IF,D).
$$
 Hence condition~(2) holds if and only if the $R/I$\+module $F/IF$
is projective.

 (3)~$\Longrightarrow$~(2) is trivial, since every $R/I$\+module
is an $I$\+contramodule $R$\+module.

 (2)~$\Longrightarrow$~(3) is an ``obtainability'' argument going back
to~\cite[proof of Theorem~9.5]{Pcta} and subsequently utilized
in the paper~\cite{PSl}.
 One proves that all the $I$\+contramodule $R$\+modules can be obtained
from $R/I$\+modules, in the relevant sense.
 Essentially, the class of all $R$\+modules $C$ satisfying~(3) for
a fixed $R$\+module $F$ is closed under certain operations, which are
listed in~\cite[Lemma~3.2 or Definition~3.3]{PSl}.
 One shows that all $I$\+contramodule $R$\+modules can be ``obtained''
from $R/I$\+modules using these operations; this is the assertion
of~\cite[Lemma~8.2]{PSl}.

 To spell out a specific argument, one can start by observing that
all $I$\+contramodule $R$\+modules are obtainable as extensions of
quotseparated $I$\+contramodule $R$\+modules
(by Proposition~\ref{extension-of-two}).
 Furthermore, any quotseparated $I$\+contramodule is obtainable as
the cokernel of an injective morphism of separated $I$\+contramodules.
 The latter are the same thing as $I$\+adically separated and
complete $R$\+modules.
 They are obtainable as transfinitely iterated extensions, in
the sense of projective limit, of $R/I$\+modules; see, e.~g.,
\cite[Proposition~18]{ET}, \cite[Lemma~9.7]{Pcta},
or~\cite[Lemma~B.10.3]{Pweak}.
\end{proof}

\begin{thm} \label{adically-flat-theorem}
 Let $X$ be an infinite set and $Q_X=\R[[X]]=\Lambda_I(R[X])$ be
the $R$\+module of decaying functions $X\rarrow\R=\Lambda_I(R)$.
 Then the $R$\+module $Q_X$ is $I$\+adically flat if and only if
the ideal $I\subset R$ is weakly proregular.
\end{thm} 

\begin{proof}
 The ``if'' assertion is a particular case of~\cite[Theorem~1.6(1)
or Theorem~6.9]{Yek2}.
 One can also obtain it by reversing the arguments in the proof of
the ``only if'' assertion that follows below.

 The argument is somewhat similar to the proof of
Proposition~\ref{decaying-functions-not-idealistic-complete}.
 Assume that, for one particular infinite set $X$, the $R$\+module
$Q_X$ is $I$\+adically flat.
 Passing to a direct summand, we can assume the set $X$ to be
countable.
 For every infinite set $Y$, the $R$\+module $Q_Y$ is a direct limit
of $R$\+modules isomorphic to~$Q_X$.
 Since the class of $I$\+adically flat modules is closed under
direct limits, it then follows that the $R$\+module $Q_Y$ is
$I$\+adically flat as well.
 Passing to the direct summands again, we see that all the projective
objects of the category of quotseparated $I$\+contramodules
$\sB=R\modl_{I\ctra}^\qs$ are $I$\+adically flat $R$\+modules.

 Furthermore, the $R/I$\+module $Q_Y/IQ_Y\simeq(R/I)[Y]$ is obviously
free, hence projective, for any module of decaying functions~$Q_Y$.
 Therefore, the $R/I$\+module $Q/IQ$ is projective for any
projective object $Q$ of the category $R\modl_{I\ctra}^\qs$.

 By Proposition~\ref{adically-projective}\,(1)\,$\Rightarrow$\,(3),
we can conclude that $\Ext_R^k(Q,B)=0$ for all projective quotseparated
$I$\+contramodules $Q$, all (quotseparated) $I$\+contramodule
$R$\+modules $B$, and all $k>0$.
 According to Lemma~\ref{derived-full-and-faithfulness-lemma},
it then follows that the triangulated functor
$\rho^\bs\:\sD^\bs(R\modl_{I\ctra}^\qs)\rarrow\sD^\bs(R\modl)$ is fully
faithful.
 Consequently, the ideal $I\subset R$ is weakly proregular by
Theorem~\ref{quotseparated-I-contramodule-theorem}.
\end{proof}

\begin{rem}
 The notion of $I$\+adic projectivity provided by the equivalent
conditions of Proposition~\ref{adically-projective} is modelled
after the definition of $I$\+adic flatness
in~\cite[Definition~4.2]{Yek2}.
 It is designed to be relevant for the purposes of
Theorem~\ref{adically-flat-theorem}.

 A quite different property was discussed under the name of
``adic projectivity'' in~\cite[Definition~3.16]{Yek0}
and~\cite[Definition~1.4]{PSY2}.
 What would be called ``$I$\+adically projective $R$\+modules''
in the terminology of~\cite{Yek0,PSY2} are called projective
quotseparated $I$\+contramodule $R$\+modules in the present paper.
 These are the projective objects of the abelian category
$R\modl_{I\ctra}^\qs$.

 The related flatness notion is that of a flat quotseparated
$I$\+contramodule $R$\+module~\cite[Sections~5.3 and~5.6]{PSl},
\cite[Sections~5\+-7]{PR}, \cite[Section~D.1]{Pcosh}.
 A quotseparated $I$\+contramodule $R$\+module $F$ is said to be
\emph{flat} if the $R/I^n$\+module $F/I^nF$ is flat for every
$n\ge1$.
 Unless the ring $R$ is Noetherian (cf.~\cite[Corollary~10.3]{Pcta},
\cite[Theorem~1.6\,(2)]{Yek2}), a flat quotseparated $I$\+contramodule
$R$\+module need not be a flat $R$\+module~\cite[Theorem~7.2]{Yek2}.

 Any flat quotseparated $I$\+contramodule $R$\+module is $I$\+adically
separated~\cite[Corollary~D.1.7]{Pcosh}, \cite[Corollary~6.15]{PR}
(these results are applicable in view of
Proposition~\ref{top-ring-contramodules} above).
 For any short exact sequence of quotseparated $I$\+contramodule
$R$\+modules $0\rarrow B\rarrow C\rarrow F\rarrow0$ with a flat
quotseparated $I$\+contramodule $R$\+module $F$, and for any
$I$\+torsion $R$\+module $N$, the tensor product sequence
$0\rarrow N\ot_RB\rarrow N\ot_RC\rarrow N\ot_RF\rarrow0$ is
exact~\cite[Lemma~6.10]{PR}.

 The projective quotseparated $I$\+contramodule $R$\+modules can be
characterized as follows~\cite[Corollary~D.1.10]{Pcosh},
\cite[Lemma~B.10.2]{Pweak} (see~\cite[Theorem~1.10]{PSY2} for
the Noetherian case).
 For any quotseparated $I$\+contramodule $R$\+module $F$, the following
conditions are equivalent:
\begin{enumerate}
\item $F$ is a projective object in $R\modl_{I\ctra}^\qs$;
\item the $R/I^n$\+module $F/I^nF$ is projective for every $n\ge1$;
\item $F$ is a flat quotseparated $I$\+contramodule $R$\+module
and the $R/I$\+module $F/IF$ is projective.
\end{enumerate}

 Any quotseparated $I$\+contramodule $R$\+module which is $I$\+adically
flat in the sense of~\cite[Definition~4.2]{Yek2} is flat in the sense
of the definition above in this remark; and any quotseparated
$I$\+contramodule $R$\+module which is $I$\+adically projective in
the sense of the equivalent conditions in
Proposition~\ref{adically-projective} is a projective object
in $R\modl_{I\ctra}^\qs$.
 But the converse implications only hold for a weakly proregular ideal
$I\subset R$ \,\cite[Theorem~1.6\,(1)]{Yek2}.
 Otherwise, a flat quotseparated $I$\+contramodule $R$\+module need
not be $I$\+adically flat, and a projective object of
$R\modl_{I\ctra}^\qs$ need not satisfy the equivalent conditions of
Proposition~\ref{adically-projective}, as
Theorem~\ref{adically-flat-theorem} shows.

 Yet another relevant projectivity notion is that of a projective
$I$\+contramodule $R$\+module, i.~e., a projective object in
$R\modl_{I\ctra}$.
 The classes of projective objects in $R\modl_{I\ctra}$ and
in $R\modl_{I\ctra}^\qs$ coincide if and only if the two abelian
categories coincide (see Lemmas~\ref{when-projective-quotseparated}
and~\ref{when-all-quotseparated}).
\end{rem}

 Here are the dual-analogous definition and assertion.
 We say that an $R$\+module $E$ is \emph{$I$\+adically injective} if
$\Ext_R^k(N,E)=0$ for all $I$\+torsion $R$\+modules $N$ and
all $k>0$.
 In view of the Eklof lemma~\cite[Lemma~1]{ET},
\cite[Lemma~3.6(a)]{PSl}, it suffices to check this condition for
finitely generated $I$\+torsion $R$\+modules~$N$.
 (In the terminology of~\cite[Definition~2.6.1]{SS}, such $R$\+modules
$E$ would be called ``relatively-$I$-injective''.)

 The following proposition is essentially equivalent
to~\cite[Theorem~0.3]{VY} (cf.\
Proposition~\ref{torsion-injective-not-idealistic-torsion} above);
see also~\cite[Proposition~2.6 and Corollary~3.4]{SS2}.

\begin{prop}
 The ideal $I\subset R$ is weakly proregular if and only if
the $R$\+module $E=\Gamma_I(J)$ is $I$\+adically injective for every
injective $R$\+module~$J$.
\end{prop}

\begin{proof}
 Put $\sA=R\modl$ and $\sT=R\modl_{I\tors}$.
 If the ideal $I\subset R$ is weakly proregular then,
by~\cite[Theorem~1.3]{Pmgm}, the map $\Ext^k_\sT(N,E)\rarrow
\Ext^k_\sA(N,E)$ induced by the exact inclusion of abelian categories
$R\modl_{I\tors}\rarrow R\modl$ is an isomorphism for all objects
$N$, $E\in R\modl_{I\tors}$ and all $k\ge0$.
 In the situation at hand with $E=\Gamma_I(J)$, the object $E$ is
injective in $R\modl_{I\tors}$; so $\Ext^k_\sA(N,E)=
\Ext^k_\sT(N,E)=0$ for $k>0$.

 Conversely, if $\Ext_\sA^k(N,E)=0$ for all $N\in\sT$, \
$E=\Gamma_I(J)$, and $k>0$, where $J$ ranges over all the injective
$R$\+modules, then the dual version of
Lemma~\ref{derived-full-and-faithfulness-lemma} tells us that
the functor $\mu^\bs\:\sD^\bs(R\modl_{I\tors})\rarrow\sD^\bs(R\modl)$
is fully faithful.
 Applying Theorem~\ref{I-torsion-theorem}, we conclude that
the ideal $I\subset R$ is weakly proregular.
\end{proof}

\Section{Conclusion}

 For any finitely generated ideal $I$ in a commutative ring $R$,
$R\modl_{I\ctra}^\qs\subset R\modl_{I\ctra}\subset R\modl$ are two
abelian full subcategories in the category of $R$\+modules $R\modl$.
 Trying to follow Yekutieli's suggested terminology~\cite{Yek3},
one could say that the category of $I$\+contramodule $R$\+modules
$R\modl_{I\ctra}$ is ``the category of derived $I$\+adically complete
modules in the sequential sense'', while the category of
quotseparated $I$\+contramodule $R$\+modules $R\modl_{I\ctra}^\qs$ is
``the category of derived $I$\+adically complete modules in
the idealistic sense''.

 What Yekutieli~\cite{Yek3} calls ``the category of derived
$I$\+adically complete complexes in the sequential sense'' is,
in our language, the full subcategory $\sD_{I\ctra}(R\modl)\subset
\sD(R\modl)$ of complexes of $R$\+modules with $I$\+contramodule
cohomology modules.

 What Yekutieli~\cite{Yek3} calls ``the category of derived
$I$\+adically complete complexes in the idealistic sense'' is likely
to be too small.
 We would suggest to modify the definition by considering
the derived category $\sD(R\modl_{I\ctra}^\qs)$ as the proper
version/replacement of the category of derived $I$\+adically complete
complexes in the idealistic sense.
 Then there is the triangulated functor $\rho\:\sD(R\modl_{I\ctra}^\qs)
\rarrow\sD(R\modl)$, but it is only fully faithful when the ideal $I$
is weakly proregular. {\emergencystretch=1em\par}

 As a middle ground between the above two versions of ``the category
of derived $I$\+adically complete complexes'', one can also consider
the derived category $\sD(R\modl_{I\ctra})$.
 Then there is the triangulated functor $\pi\:\sD(R\modl_{I\ctra})
\rarrow\sD(R\modl)$, which is also, generally speaking,
not fully faithful.

 To summarize the discussion of derived $I$\+adically complete
complexes and derived $I$-adic completions, there is a diagram
of triangulated functors
\begin{equation} \label{three-derived-complete-categories}
\begin{tikzcd}
\sD(R\modl_{I\ctra}^\qs) \arrow[rr, "\beta"] \arrow[rrdd, "\rho"]
&&\sD(R\modl_{I\ctra}) \arrow[rr, "\zeta"] \arrow[dd, "\pi"]
&&\sD_{I\ctra}(R\modl) \arrow[lldd, tail, "\iota"] \\ \\
&& \sD(R\modl) \arrow[lluu, bend left = 22, "\lambda"]
\arrow[uu, bend left = 30, "\delta"]
\arrow[rruu, two heads, bend left = 15, "\eta" near start]
\end{tikzcd}
\end{equation}
 Here straight arrows form a commutative diagram.
 The curvilinear arrows show left adjoint functors.
 Only the functor shown by the rightmost diagonal arrow (with a tail)
is fully faithful in general.
 The diagonal arrow with two heads on the right-hand side of
the diagram is a Verdier quotient functor.
 
 The sequential derived $I$-adic completion functor is the composition
of the two adjoint functors on the right-hand side,
$$
 \iota\circ\eta=\boR\Hom_R(K_\infty\spcheck(R;\s),{-})\:
 \sD(R\modl)\lrarrow\sD(R\modl).
$$
 The idealistic derived $I$-adic completion functor is the composition
of the two adjoint functors on the left-hand side,
$$
 \rho\circ\lambda=\boL\Lambda_I\:\sD(R\modl)\lrarrow\sD(R\modl).
$$
 
 The leftmost horizontal functor~$\beta$
in~\eqref{three-derived-complete-categories} is fully faithful
if and only it is a triangulated equivalence (and if and only if
$R\modl_{I\ctra}^\qs=R\modl_{I\ctra}$).
 Similarly, the rightmost horizontal functor~$\zeta$ is fully faithful
if and only if it is a triangulated equivalence.
 The composition of the two horizontal functors (denoted by~$\xi$
in Section~\ref{full-and-faithfulness-secn}) is fully faithful
if and only if it is a triangulated equivalence.

 If the ideal $I$~is weakly proregular, then both horizontal functors
in~\eqref{three-derived-complete-categories} are triangulated
equivalences.
 If the ideal $I$~is not weakly proregular, then the composition
$\xi=\zeta\circ\beta$ of the two horizontal functors is not
an equivalence (but one of them can be).
 See Remark~\ref{main-remark} for further discussion.

\medskip

 The situation with derived torsion complexes is similar but simpler.
 There is only one abelian category of $I$\+torsion $R$\+modules,
$R\modl_{I\tors}\subset R\modl$.

 What is called ``the category of derived $I$\+torsion complexes in
the sequential sense'' in~\cite{Yek3} is, in our language, the full
subcategory $\sD_{I\tors}(R\modl)\subset\sD(R\modl)$ of complexes of
$R$\+modules with $I$\+torsion cohomology modules.

 What is called ``the category of derived $I$\+torsion complexes in
the idealistic sense'' in~\cite{Yek3} is likely to be too small.
 We would suggest to modify the definition by considering the derived
category $\sD(R\modl_{I\tors})$ as the proper version/replacement of
the category of derived $I$\+torsion complexes in the idealistic sense.
 Then there is the triangulated functor $\mu\:\sD(R\modl_{I\tors})
\rarrow\sD(R\modl)$, but it is only fully faithful when the ideal $I$
is weakly proregular.

 To summarize the discussion of derived $I$\+torsion complexes and
derived $I$-torsion functors, there is a diagram of triangulated
functors
\begin{equation} \label{two-derived-torsion-categories}
\begin{tikzcd}
\sD(R\modl_{I\tors}) \arrow[rrrr,"\chi"] \arrow[rrdd, "\mu"']
&&&&\sD_{I\tors}(R\modl) \arrow[lldd, tail, "\upsilon"'] \\ \\
&& \sD(R\modl) \arrow[lluu, bend right = 16, "\gamma"' near start]
\arrow[rruu, two heads, bend right = 21, "\theta"']
\end{tikzcd} 
\end{equation}
 Here straight arrows form a commutative triangular diagram.
 The curvilinear arrows show right adjoint functors.
 The diagonal arrow with a tail on the right-hand side shows a fully
faithful functor.
 The rightmost diagonal arrow with two heads shows a Verdier quotient
functor.

 The sequential derived $I$-torsion functor is the composition of
the two adjoint functors on the right-hand side,
$$
 \upsilon\circ\theta=K_\infty\spcheck(R;\s)\ot_R{-}\:
 \sD(R\modl)\lrarrow\sD(R\modl).
$$
 The idealistic derived $I$-torsion functor is the composition of
the two adjoint functors on the left-hand side,
$$
 \mu\circ\gamma=\boR\Gamma_I\:\sD(R\modl)\lrarrow\sD(R\modl).
$$

 The horizontal functor~$\chi$
in~\eqref{two-derived-torsion-categories}
is fully faithful if and only if it is a triangulated equivalence,
and if and only if the ideal $I\subset R$ is weakly proregular.

\end{document}